\documentclass[11pt]{amsart}
\usepackage{amsmath,amssymb,amsthm,mathrsfs}
\usepackage{geometry,booktabs}
\usepackage[colorlinks=true,linkcolor=blue,citecolor=blue,urlcolor=blue]{hyperref}
\usepackage{enumitem}

\newtheorem{theorem}{Theorem}[section]
\newtheorem{proposition}[theorem]{Proposition}
\newtheorem{lemma}[theorem]{Lemma}
\newtheorem{corollary}[theorem]{Corollary}

\theoremstyle{definition}
\newtheorem{definition}[theorem]{Definition}
\theoremstyle{remark}
\newtheorem{remark}[theorem]{Remark}

\newcommand{\eps}{\varepsilon}
\newcommand{\id}{\mathrm{id}}
\newcommand{\cop}{\mathrm{cop}}
\newcommand{\op}{\mathrm{op}}
\newcommand{\Comod}{\operatorname{Comod}_{r}}
\newcommand{\can}{\operatorname{can}}
\newcommand{\Hom}{\operatorname{Hom}}
\newcommand{\Alg}{\operatorname{Alg}}

\newcommand{\Aut}{\operatorname{Aut}}
\newcommand{\EndHopf}{\operatorname{End}_{\mathrm{Hopf}}}
\newcommand{\AutHopf}{\operatorname{Aut}_{\mathrm{Hopf}}}

\numberwithin{equation}{section}
\allowdisplaybreaks

\title{Right Coideal Subalgebras, Hopf--Galois Objects, and $\imath$-Hopf Algebras}
\author{Rongchuan Xiong}
\address{Department of Mathematics, Changzhou University, Changzhou 213164, China}
\email{rcxiong@foxmail.com}

\makeatletter
\@namedef{subjclassname@2020}{\textup{2020} Mathematics Subject Classification}
\makeatother

\date{}

\begin{document}
	\begin{abstract}
		We study right comodule algebras and their connections with Hopf bi-Galois objects and right coideal subalgebras. For a Hopf $2$-cocycle $\sigma$ on $K$ and a unital
		right-colinear algebra map $\Gamma:A\to K$, we show that
		$\Gamma:A_\sigma\to K^\sigma$ is an algebra map precisely
		when $\sigma$ is trivial on
		$\operatorname{im}\Gamma\otimes\operatorname{im}\Gamma$. We describe comodule-algebra morphisms
		$A_\sigma\to K^\sigma$ by characters and establish
		injectivity criteria using the coradical and extensions
		of characters to convolution-invertible functionals. For coactions on Hopf quotients projecting to the coproduct, character-evaluation maps are injective, and equality of their images is determined by a unique algebra character.
		
		For the $\imath$-Hopf algebras $H^\imath_\tau$ of Chen--Lu--Pan--Ruan--Wang, with $H$ having a bijective antipode and a Hopf pairing, we obtain a Hopf bi-Galois realization $H^\imath_\tau\hookrightarrow F_\sigma(H\otimes H^{\mathrm{cop}})$. We identify algebra characters with nonzero $\tau$-twisted compatible maps $\chi$ and prove that the corresponding morphisms $\xi_{\tau,\chi}$ of Chen--Lu--Pan--Ruan--Wang are injective, yielding right coideal realizations in $(H\otimes H)^\imath$. We classify the $H^\imath_\tau$ up to comodule-algebra isomorphism and determine when these realizations have the same image.
	\end{abstract}
	
	\maketitle

	\section{Introduction}\label{sec:introduction}
	
	Hopf bi-Galois objects provide a concrete realization of monoidal
	equivalences between comodule categories through cotensor products
	\cite[Theorem~3.2.2]{Sch04}. In the cleft case, Hopf
	$2$-cocycles make this construction explicit: a Hopf $2$-cocycle
	$\sigma$ on a Hopf algebra $K$ produces both the Hopf $2$-cocycle deformation
	$K^\sigma$ and a $(K,K^\sigma)$-Hopf bi-Galois object $F_\sigma(K)$
	\cite[Proposition~3.1.6]{Sch04}.

	We study morphisms from right comodule algebras to
	Hopf algebras, their images as right coideal subalgebras,
	and their behavior under cocycle deformation. Let $K$ be a Hopf algebra, $\sigma$ a Hopf $2$-cocycle
	on $K$, and $(A,\rho)$ a right $K$-comodule algebra.
	For a morphism $\Gamma:(A,\rho)\longrightarrow(K,\Delta_K)$ in $\Alg(\Comod(K))$,
	the monoidal equivalence
	$F_\sigma:\Comod(K)\to\Comod(K^\sigma)$
	induces a morphism
	\[
	\Gamma:A_\sigma\longrightarrow F_\sigma(K)
	\quad\text{in }\Alg(\Comod(K^\sigma)).
	\]
	Writing $B=\operatorname{im}\Gamma$, one has $A_\sigma/\ker\Gamma\cong B_\sigma$
	as right $K^\sigma$-comodule algebras.
	Conversely, every injective unital algebra map
	$\Gamma:A\to K$ whose image is a right coideal
	subalgebra determines such a right $K$-comodule
	algebra structure on $A$.

	We ask when $\Gamma:A_\sigma\to K^\sigma$ is an algebra map.
	Proposition~\ref{prop:Gamma-hopf-target} shows that this happens
	if and only if
	\[
	\sigma|_{B\otimes B}=\eps_K|_B\otimes\eps_K|_B.
	\]
	
	A second question is how algebra characters behave under this
	deformation. Let $C$ be a right $L$-comodule algebra with coaction
	$\rho_C$. For every $t\in\Alg(C,k)$, the map
	$j_t=(t\otimes\id)\rho_C:C\to L$ is a morphism in
	$\Alg(\Comod(L))$, and every such morphism arises uniquely in
	this way. We refer to $j_t$ as the character-evaluation map
	associated with $t$. Its image is a right coideal subalgebra of
	$L$. For related uses of characters in specific settings, see Kolb--Stokman \cite[Section~2.6]{KoSt09} on covariantized algebras and Etingof--Walton \cite[Lemma~3.2]{EW14} on Hopf actions on commutative domains.
	
	For the deformation above, this gives a bijection
	\[
	\Alg(A_\sigma,k)
	\;\xrightarrow{\sim}\;
	\Hom_{\Alg(\Comod(K^\sigma))}
	\bigl(A_\sigma,K^\sigma\bigr).
	\]
	In particular, $\Gamma:A_\sigma\to K^\sigma$ is a morphism in
	$\Alg(\Comod(K^\sigma))$ precisely when $\eps_K\Gamma$ is a
	character of $A_\sigma$. Thus the restriction condition on
	$\sigma$ is equivalent to this character condition. More
	generally, for $r\in\Alg(A,k)$,
	Corollary~\ref{cor:deformed-characters} shows that $r$ remains
	a character of $A_\sigma$ exactly when $\sigma$ is trivial on
	$\operatorname{im}j_r\otimes\operatorname{im}j_r$.
	
	We next ask when the character-evaluation map is injective. For a
	right coideal subalgebra $B\subseteq K$ and
	$t\in\Alg(B_\sigma,k)$,
	Theorem~\ref{thm:coideal-character-injectivity} shows that
	$j_t:B_\sigma\to K^\sigma$ is injective exactly when its
	restriction to $B\cap K_0$ is injective, where $K_0$ is the
	coradical of $K$; equivalently, $t$ extends to a
	convolution-invertible functional on $K$. Cocycle deformation
	also preserves the two-sided ideals that are subcomodules. Hence,
	if $A$ has no nonzero proper such ideal, every character-evaluation
	map of $A_\sigma$ is injective, see
	Proposition~\ref{prop:stable-ideals}.
	
	These three questions are genuinely distinct. The inclusion
	$B_\sigma\to K^\sigma$ may fail to be an algebra map, characters
	may disappear under deformation, and an existing character may
	have a noninjective character-evaluation map; see
	Remarks~\ref{rem:hopf-target-closure},
	\ref{rem:twisted-characters}, and~\ref{rem:character-kernel},
	respectively.

	We also consider when two character-evaluation maps have
	the same image. For injective character-evaluation maps
	$j_t:A_\sigma\to K^\sigma$ and
	$j_u:A'_\sigma\to K^\sigma$, their images coincide
	if and only if there is an isomorphism
	$T:A_\sigma\to A'_\sigma$ in
	$\Alg(\Comod(K^\sigma))$ satisfying $u\circ T=t$.
	Such an isomorphism is unique, with
	$T=j_u^{-1}\circ j_t$, where the inverse is taken
	on the common image.
	
	Let $\pi:K\to Q$ be a surjective Hopf algebra map,
	and let $(Q,\rho),(Q,\rho')\in\Alg(\Comod(K))$ satisfy
	\[
	(\id_Q\otimes\pi)\rho
	=(\id_Q\otimes\pi)\rho'=\Delta_Q.
	\]
	Lemma~\ref{lem:quotient-character-injectivity} shows
	that every character-evaluation map of either
	deformed comodule algebra is injective.
	Lemma~\ref{lem:quotient-coaction-morphisms} further
	shows that every comodule-algebra isomorphism between
	them is uniquely of the form
	\[
	U_r(q)=r(q_1)q_2,\qquad r\in\Alg(Q,k).
	\]
	Consequently, for
	$t\in\Alg((Q,\rho)_\sigma,k)$ and
	$u\in\Alg((Q,\rho')_\sigma,k)$,
	Theorem~\ref{thm:character-evaluation-images}
	gives
	\[
	\operatorname{im}j_t^\rho
	=\operatorname{im}j_u^{\rho'}
	\iff
	\begin{cases}
		\rho'=(U_r\otimes\id_K)\rho U_{r^{-1}},\\
		u=r^{-1}*t
	\end{cases}
	\]
	for a unique $r\in\Alg(Q,k)$, where
	$r^{-1}=r\circ S_Q$ and convolution uses $\Delta_Q$.
	The same theorem describes these equalities through
	a free action of $\Alg(Q,k)^{\op}$.

	These criteria admit concrete descriptions for right coideal
	subalgebras of group algebras, in terms of group cocycles, monoid
	characters, and conjugation stability; see
	Proposition~\ref{prop:group-algebras} and
	Corollary~\ref{cor:group-MU-stability}.
	
	Our main application concerns the $\imath$-Hopf algebras of
	Chen--Lu--Pan--Ruan--Wang \cite{CLPRW}, whose construction
	motivates our Hopf-algebraic approach. Suppose that $H$ is a Hopf
	algebra with bijective antipode $S$, that $\phi:H\otimes H\to k$
	is a Hopf pairing, and that $\tau:H\to H$ is a Hopf endomorphism.
	Set $K=H\otimes H^{\cop}$. The endomorphism $\tau$ determines the
	right $K$-comodule algebra $(H,\delta_\tau)$ and the injective
	morphism in $\Alg(\Comod(K))$:
	\[
	\Gamma_\tau:(H,\delta_\tau)\hookrightarrow (K,\Delta_K),
	\qquad
	\Gamma_\tau(h)=\tau(h_2)\otimes h_1,
	\]
	while the Hopf pairing $\phi$ determines a Hopf $2$-cocycle
	$\sigma$ on $K$. We prove that
	\[
	F_\sigma(H,\delta_\tau)\cong H^\imath_\tau
	\qquad\text{in}\ \Alg(\Comod(K^\sigma)).
	\]
	In particular, the multiplication
	\[
	a*_\tau b
	=
	\phi(\tau(b_2),a_1)a_2b_1
	\]
	of $H^\imath_\tau$ is precisely the multiplication induced by the
	monoidal functor $F_\sigma$. Applying $F_\sigma$ to $\Gamma_\tau$
	then gives an injective morphism
	\[
	H^\imath_\tau\hookrightarrow F_\sigma(K)
	\qquad\text{in}\ \Alg(\Comod(K^\sigma)).
	\]
	Thus $H^\imath_\tau$ embeds as a right $K^\sigma$-comodule
	subalgebra of the Hopf bi-Galois object $F_\sigma(K)$. By
	Proposition~\ref{prop:Gamma-hopf-target}, the map
	$\Gamma_\tau:H^\imath_\tau\to K^\sigma$ is an algebra map if and
	only if $\phi(\tau(b),a)=\eps(a)\eps(b)$ for all $a,b\in H$, see
	Corollary~\ref{cor:graph-hopf-target}.
	
	Moreover, Proposition~\ref{prop:cocycle-double} gives a Hopf
	algebra isomorphism $K^\sigma\cong D(H)$. Composing with the Hopf
	algebra isomorphism
	\[
	\Phi^{\mathrm D}:D(H)\xrightarrow{\sim}(H\otimes H)^\imath
	\]
	of \cite[Proposition~2.8]{CLPRW} gives
	$K^\sigma\xrightarrow{\sim}(H\otimes H)^\imath$. By
	Lemma~\ref{lem:coaction-comparison}, under this identification
	the $K^\sigma$-coaction on $H^\imath_\tau$ coincides with the
	$(H\otimes H)^\imath$-coaction $\Psi$ in \eqref{eq:Psi}. Thus,
	under $\Phi^{\mathrm D}$, the same Hopf $2$-cocycle deformation
	recovers both the Hopf algebra structure on $(H\otimes H)^\imath$
	and its coaction on $H^\imath_\tau$. We show that
	$B_\sigma=\Gamma_\tau(H^\imath_\tau)\subseteq F_\sigma(K)$ is
	stable under the right Miyashita--Ulbrich action of $K^\sigma$ if
	and only if $\tau(H)\subseteq Z(H)$, see
	Proposition~\ref{prop:Gamma-MU-stability}.
	
	The comodule algebra $(H,\delta_\tau)$ has no nonzero proper
	two-sided ideal that is a subcomodule, see
	Lemma~\ref{lem:graph-stable-ideals}. This property is preserved by
	cocycle deformation, so every character-evaluation map of
	$H^\imath_\tau$ is injective, by
	Proposition~\ref{prop:stable-ideals}.
	
	This point of view also relates $\tau$-twisted compatible maps
	to algebra characters. We show that the nonzero $\tau$-twisted
	compatible maps $\chi:H\to k$ of
	\cite[Definition~2.9]{CLPRW} correspond to characters
	$t_\chi\in\Alg(H^\imath_\tau,k)$, see
	Theorem~\ref{thm:compatible-characters}. Under this
	correspondence, the morphism
	$\xi_{\tau,\chi}:H^\imath_\tau\to(H\otimes H)^\imath$ of
	\cite[Proposition~2.10]{CLPRW} is given by
	\[
	\xi_{\tau,\chi}=(t_\chi\otimes\id)\Psi.
	\]
	Consequently, $\xi_{\tau,\chi}$ is injective for every nonzero
	$\tau$-twisted compatible map, and its image is a right coideal
	subalgebra of $(H\otimes H)^\imath$, see
	Theorem~\ref{thm:xi-injective}.
	Proposition~\ref{prop:compatible-convolution-inverse} also gives
	explicit convolution inverses of compatible maps and linear left
	inverses of the corresponding morphisms.
	
	We further classify the $H^\imath_\tau$ up to isomorphism in
	$\Alg(\Comod(K^\sigma))$.
	Proposition~\ref{prop:ihopf-isomorphisms} shows that
	$\Alg(H,k)$ acts on the Hopf endomorphisms of $H$ by
	\[
	r\cdot\tau=r*\tau*r^{-1},
	\]
	and the isomorphism classes are precisely its orbits. For nonzero
	$\tau$- and $\upsilon$-twisted compatible maps $\chi$ and $\psi$,
	respectively, Corollary~\ref{cor:compatible-equal-images} shows
	that
	\[
	\operatorname{im}\xi_{\tau,\chi}
	=
	\operatorname{im}\xi_{\upsilon,\psi}
	\]
	if and only if there is a unique $r\in\Alg(H,k)$ satisfying
	\[
	\upsilon=r*\tau*r^{-1},
	\qquad
	\psi=\chi*r^{-1}.
	\]
	Here convolution is taken with respect to $\Delta_H$.
	
	In the Borel case of \cite[Section~4.1]{CLPRW}, the
	$\tau$-twisted compatible map $\chi$ of
	\cite[Lemma~4.2]{CLPRW} gives the algebra embedding
	$\xi_{\tau,\chi}$ of \cite[Lemma~4.3]{CLPRW},
	and $H^\imath_\tau$ is identified with the
	quasi-split universal $\imath$-quantum group in
	\cite[Theorem~4.4]{CLPRW}.
	For the same Borel Hopf algebra with its fixed
	Hopf pairing, Corollary~\ref{cor:borel-specialization}
	shows that distinct involutions of the Cartan
	matrix preserving the symmetrizer give
	nonisomorphic right $K^\sigma$-comodule algebras
	$H^\imath_\tau$, each having only the identity
	comodule-algebra automorphism.
	Moreover, for nonzero $\tau$- and
	$\upsilon$-twisted compatible maps $\chi,\psi$,
	respectively,
	\[
	\operatorname{im}\xi_{\tau,\chi}
	=\operatorname{im}\xi_{\upsilon,\psi}
	\iff
	\tau=\upsilon,\qquad \chi=\psi.
	\]
	
	The paper is organized as follows. Section~\ref{sec:cocycles} recalls basic notions of Hopf bi-Galois objects and cocycle deformations. Section~\ref{sec:coideal-deformation} studies right coideal subalgebras of the associated Hopf bi-Galois object, the behavior of algebra characters under cocycle deformation, character-evaluation maps, injectivity criteria, equality of images, and comodule subalgebras. Sections~\ref{sec:ihopf} and~\ref{sec:realization} apply these results to $\imath$-Hopf algebras. Section~\ref{sec:characters} studies $\tau$-twisted compatible maps, the resulting injective algebra maps whose images are right coideal subalgebras, equality of their images, and explicit formulas for convolution inverses.

	\paragraph{\textbf{Notation.}}
	Throughout, $k$ is a field, and all tensor products are over
	$k$. For a Hopf algebra $K$, let $\Comod(K)$ denote the monoidal
	category of right $K$-comodules and $\Alg(\Comod(K))$ its
	category of algebra objects. For algebras $R,T$, let
	$\Alg(R,T)$ be the set of unital algebra maps. Thus, for
	$B,C\in\Alg(\Comod(K))$,
	\[
	\Hom_{\Alg(\Comod(K))}(B,C)
	=
	\{f\in\Alg(B,C)\mid
	\rho_C f=(f\otimes\id_K)\rho_B\}.
	\]
	
	We use Sweedler notation, omitting summation signs:
	$\Delta(c)=c_1\otimes c_2$ and $\rho_A(a)=a_0\otimes a_1$.
	
	For a coalgebra $C$ and an algebra $R$, convolution on
	$\Hom(C,R)$ is given by
	$(f*g)(c)=f(c_1)g(c_2)$ and
	$1_{\mathrm{conv}}=\eta_R\circ\eps_C$. For a
	convolution-invertible map $f$, we write $f^{-1}$ for its
	convolution inverse; inverses of linear isomorphisms refer to
	composition. Scalar-valued maps are regarded as $R$-valued
	maps via $\eta_R:k\to R$ when taking convolution.

	\section{Hopf bi-Galois objects and cocycle deformations}
	\label{sec:cocycles}
	This section collects the background on Hopf Galois and bi-Galois
	objects and on Hopf $2$-cocycle deformations that will be used
	below. For Hopf bi-Galois objects and their monoidal equivalences,
	we refer to \cite{Sch96,Sch04}; our conventions for Hopf
	$2$-cocycles and the explicit cleft realization follow
	\cite[Section~1]{KS05}. For earlier work on Hopf Galois extensions
	and Galois-type correspondences, see
	\cite{VanOystaeyenZhang94} and references therein.

	\subsection{Hopf Galois and bi-Galois objects}
	
	For a right $K$-comodule algebra $A$ with coaction
	$\rho:A\to A\otimes K$, its coinvariants are
	\[
	A^{\operatorname{co}K}
	=
	\{a\in A\mid \rho(a)=a\otimes1\}.
	\]

	\begin{definition}
		A \emph{right $K$-Galois object} is a nonzero right
		$K$-comodule algebra $A$ such that
		$A^{\operatorname{co}K}=k1$ and the canonical map
		\[
		\can_R:A\otimes A\longrightarrow A\otimes K,
		\qquad
		a\otimes b\longmapsto ab_0\otimes b_1
		\]
		is bijective.
		
		It is \emph{cleft} if there exists a right $K$-colinear
		convolution-invertible map $K\to A$
		\cite{DT86}.
		
		Let $L$ and $K$ be Hopf algebras. An
		$(L,K)$-\emph{Hopf bi-Galois object} is an
		$L$-$K$-bicomodule algebra which is simultaneously a left
		$L$-Galois object and a right $K$-Galois object
		\cite[Definition~3.1.1]{Sch04}.
	\end{definition}
	
	For later use, if $A$ is a right $K$-Galois object, write
	\[
	\can_R^{-1}(1\otimes h)
	=
	h^{[1]}\otimes h^{[2]},
	\qquad h\in K.
	\]
	Since $A^{\operatorname{co}K}=k1$, the right
	Miyashita--Ulbrich action is defined on $A$ by
	\[
	a\leftharpoonup h
	=
	h^{[1]}ah^{[2]}.
	\]
	For the regular right $K$-Galois object $A=K$, this is the right
	adjoint action $a\leftharpoonup h=S(h_1)ah_2.$
	We use this convention below; see e.g. \cite[Section 2.1]{Sch04}.
	
	For a right $K$-comodule $V$ and a left $K$-comodule $W$, with
	coactions $\rho_V:V\to V\otimes K$ and $\lambda_W:W\to K\otimes W$,
	their cotensor product is
	\[
	V\square_K W
	=
	\ker\bigl(
	\rho_V\otimes\id_W-\id_V\otimes\lambda_W
	\bigr)
	\subseteq V\otimes W.
	\]
	
	An $(L,K)$-Hopf bi-Galois object $A$ induces, by cotensor product,
	a monoidal equivalence
	\[
	-\square_L A:
	\Comod(L)\longrightarrow\Comod(K),
	\]
	see e.g. \cite[Theorem~3.2.2]{Sch04}. For comodule algebras, the
	multiplication on the cotensor product is inherited from the
	tensor product algebra. We use below the explicit Hopf
	bi-Galois object arising from a Hopf $2$-cocycle.
	
	\subsection{Hopf \texorpdfstring{$2$}{2}-cocycle deformations}
	
	\begin{definition}
		Let $K$ be a Hopf algebra. A \emph{Hopf $2$-cocycle} on $K$
		is a normalized convolution-invertible linear map
		$\sigma:K\otimes K\to k$ satisfying
		\begin{align}
			\sigma(x_1,y_1)\sigma(x_2y_2,z)
			&=
			\sigma(y_1,z_1)\sigma(x,y_2z_2),
			\label{eq:cocycle}\\
			\sigma(1,x)&=\sigma(x,1)=\eps(x).
		\end{align}
	\end{definition}
	
	Thus
	\[
	\sigma(x_1,y_1)\sigma^{-1}(x_2,y_2)
	=
	\sigma^{-1}(x_1,y_1)\sigma(x_2,y_2)
	=
	\eps(x)\eps(y).
	\]
	
	\begin{proposition}[Hopf $2$-cocycle deformation]
		\label{prop:transport}
		Let $\sigma$ be a Hopf $2$-cocycle on $K$.
		
		The vector space $K$, with its original coalgebra structure and
		multiplication
		\begin{equation}\label{eq:hopf-product}
			x\star_\sigma y
			=
			\sigma(x_1,y_1)x_2y_2\sigma^{-1}(x_3,y_3),
		\end{equation}
		is a Hopf algebra, denoted by $K^\sigma$. We call $K^\sigma$
		the \emph{Hopf $2$-cocycle deformation} of $K$.
		
		If $B$ is a right $K$-comodule algebra with coaction
		$\rho_B:B\to B\otimes K$, written $\rho_B(b)=b_0\otimes b_1$,
		then $B$, with the same coaction
		and multiplication
		\begin{equation}\label{eq:transport}
			b\cdot_\sigma c
			=
			b_0c_0\sigma^{-1}(b_1,c_1),
		\end{equation}
		is a right $K^\sigma$-comodule algebra, denoted by $B_\sigma$. We call $B_\sigma$ the \emph{comodule-algebra deformation} of $B$.
	\end{proposition}
	
	\begin{proof}
		These are the standard cocycle-deformation formulas; see e.g. 
		\cite[Section~1, equations~(1.1)--(1.3)]{KS05}.
	\end{proof}
	
	Since $K$ and $K^\sigma$ have the same coalgebra, the identity
	on underlying vector spaces, coactions, and linear maps defines a
	monoidal equivalence
	\[
	F_\sigma:\Comod(K)\longrightarrow\Comod(K^\sigma),
	\]
	with tensor constraint on right $K$-comodules $V,W$
	\begin{equation}\label{eq:tensorator}
		F_\sigma(V)\otimes F_\sigma(W)
		\longrightarrow
		F_\sigma(V\otimes W),
		\qquad
		v\otimes w\longmapsto v_0\otimes w_0\sigma^{-1}(v_1,w_1).
	\end{equation}
	The induced equivalence
	$F_\sigma:\Alg(\Comod(K))\longrightarrow\Alg(\Comod(K^\sigma))$
	satisfies $F_\sigma(B)=B_\sigma$.
	
	The cocycle identity gives the monoidal coherence of
	\eqref{eq:tensorator}. The convolution inverse $\sigma^{-1}$ is a
	Hopf $2$-cocycle on $K^\sigma$, and the corresponding deformation
	recovers $K$ and the original comodule-algebra multiplication;
	see e.g. \cite[Section~1]{KS05}; thus the inverse equivalence is
	\[
	F_{\sigma^{-1}}:
	\Comod(K^\sigma)\longrightarrow\Comod(K).
	\]
	
	The Hopf algebra $K^\sigma$ has the two-sided multiplication
	\eqref{eq:hopf-product}, whereas $B_\sigma$ has the one-sided
	multiplication \eqref{eq:transport}. In particular, regarding $K$
	as its regular right $K$-comodule algebra gives the algebra
	underlying the associated Hopf bi-Galois object:
	\begin{equation}\label{eq:regular-deformation}
		F_\sigma(K)=K_{\sigma}:
		\qquad
		x\cdot_\sigma y
		=
		x_1y_1\sigma^{-1}(x_2,y_2).
	\end{equation}
	
	For left comodules, the corresponding multiplication uses $\sigma$
	on the left coaction coefficients; see e.g. \cite[Section~1.3]{GKM12} and
	\cite[Section~2.1]{Mombelli10}.
	
	\subsection{The associated Hopf bi-Galois object}
	
	\begin{proposition}\label{prop:bigalois}
		Let $\sigma$ be a Hopf $2$-cocycle on $K$.
		The algebra $F_\sigma(K)$ of
		\eqref{eq:regular-deformation} is a
		$(K,K^\sigma)$-Hopf bi-Galois object. Its left and right
		coactions are induced by the coproduct:
		\[
		\lambda(x)=x_1\otimes x_2
		\in K\otimes F_\sigma(K),
		\qquad
		\rho(x)=x_1\otimes x_2
		\in F_\sigma(K)\otimes K^\sigma.
		\]
	\end{proposition}
	
	\begin{proof}
		The convolution inverse $\sigma^{-1}$ is a Hopf $2$-cocycle
		on $K^\sigma$. Applying
		\cite[Proposition~3.1.6]{Sch04} to
		$(K^\sigma,\sigma^{-1})$ gives a
		$(K,K^\sigma)$-Hopf bi-Galois object. Its one-sided
		multiplication is
		\[
		\begin{aligned}
			\sigma^{-1}(x_1,y_1)(x_2\star_\sigma y_2)
			&=
			\sigma^{-1}(x_1,y_1)\sigma(x_2,y_2)
			x_3y_3\sigma^{-1}(x_4,y_4)\\
			&=
			x_1y_1\sigma^{-1}(x_2,y_2),
		\end{aligned}
		\]
		where the last equality uses $\sigma^{-1}*\sigma=\eps_K\otimes\eps_K$.
		Thus the algebra is precisely $F_\sigma(K)$.
	\end{proof}
	
	The crossed-product description in the proof shows that
	$F_\sigma(K)$ is cleft as a right $K^\sigma$-Galois object.
	The cotensor functor $-\square_KF_\sigma(K)$ realizes the monoidal
	equivalence $F_\sigma$; see \cite[Theorem~3.2.2]{Sch04}.
	For every right $K$-comodule $V$, the map
	\[
	F_\sigma(V)\xrightarrow{\sim}V\square_KF_\sigma(K),
	\qquad v\longmapsto v_0\otimes v_1,
	\]
	is a natural isomorphism of right $K^\sigma$-comodules, with inverse
	$\id\otimes\eps_K$ by coassociativity and the counit identities.
	The tensor constraint induced by multiplication in $F_\sigma(K)$
	agrees with \eqref{eq:tensorator} under these identifications.
	For a right $K$-comodule algebra $B$, this gives $B_\sigma\cong B\square_KF_\sigma(K)$
	as right $K^\sigma$-comodule algebras: applying $\id\otimes\eps_K$
	to the inherited product gives
	\[
	b_0c_0\,\eps_K(b_1\cdot_\sigma c_1)
	=b_0c_0\sigma^{-1}(b_1,c_1).
	\]

	\section{Right coideal subalgebras and Hopf bi-Galois objects}
	\label{sec:coideal-deformation}
	In this section, we study comodule-algebra morphisms into Hopf algebras, their images as right coideal subalgebras, and the behavior of character-evaluation maps under Hopf $2$-cocycle deformation. Suppose that $K$ is a Hopf algebra with coradical $K_0$, $\sigma$ is a
	normalized convolution-invertible Hopf $2$-cocycle on $K$, and
	$(A,\rho)$ is  a right $K$-comodule algebra. Let
	$\Gamma:(A,\rho)\to(K,\Delta_K)$ be a morphism in
	$\Alg(\Comod(K))$. Set $B=\operatorname{im}\Gamma$ and
	$I=\ker\Gamma$. 
	\begin{lemma}\label{lem:deformation-ideals}
		Let $J\subseteq A$ be a right $K$-subcomodule.
		Then $J$ is a two-sided ideal of $A$ if and only if
		it is a two-sided ideal of $A_\sigma$. In this case,
		there is an isomorphism
		\[
		(A/J)_\sigma\cong A_\sigma/J
		\]
		of right $K^\sigma$-comodule algebras.
	\end{lemma}
	
	\begin{proof}
		If $J$ is an ideal of $A$, the condition
		$\rho(J)\subseteq J\otimes K$ gives
		\[
		A\cdot_\sigma J\subseteq J,
		\qquad
		J\cdot_\sigma A\subseteq J.
		\]
		The converse follows by applying the inverse
		deformation $F_{\sigma^{-1}}$.
		Applying $F_\sigma$ to the quotient morphism
		$A\to A/J$ yields a surjective morphism with
		kernel $J$, and hence the stated isomorphism.
		See also \cite[proof of Lemma~3.1]{GKM12}
		for the left-comodule case.
	\end{proof}
	
	\begin{lemma}\label{lem:coideal-deformation}
		The image $B$ is a unital right coideal subalgebra
		of $K$, and $I$ is a right $K$-subcomodule ideal
		of $A$. Moreover,
		\[
		\Gamma:A_\sigma\longrightarrow F_\sigma(K)
		\]
		is a morphism in $\Alg(\Comod(K^\sigma))$
		with kernel $I$ and image $B_\sigma$, inducing
		an isomorphism
		\[
		A_\sigma/I\xrightarrow{\sim}B_\sigma,
		\qquad a+I\longmapsto\Gamma(a),
		\]
		of right $K^\sigma$-comodule algebras.
	\end{lemma}
	\begin{proof}
		Since $\Gamma$ is a unital colinear algebra map,
		$B$ is a right coideal subalgebra of $K$, and
		$I$ is a right $K$-subcomodule ideal of $A$.
		By Lemma~\ref{lem:deformation-ideals},
		$I$ is also an ideal of $A_\sigma$.
		For $a,b\in A$,
		\begin{equation}\label{eq:deformed-Gamma}
			\Gamma(a\cdot_\sigma b)
			=
			\Gamma(a_0)\Gamma(b_0)\,\sigma^{-1}(a_1,b_1)
			=
			\Gamma(a)_1\Gamma(b)_1\,
			\sigma^{-1}\bigl(\Gamma(a)_2,\Gamma(b)_2\bigr),
		\end{equation}
		so $\Gamma:A_\sigma\to F_\sigma(K)$ is an
		algebra map. Since $F_\sigma$ preserves
		the underlying linear maps and coactions,
		its kernel and image are $I$ and $B_\sigma$,
		respectively. The asserted isomorphism follows.
	\end{proof}
	
	\begin{proposition}
		\label{prop:Gamma-hopf-target}
		With the notation of Lemma~\ref{lem:coideal-deformation},
		$\Gamma:A_\sigma\to K^\sigma$ is an algebra map if and only if
		\begin{equation}\label{eq:Gamma-hopf-target}
			\sigma|_{B\otimes B}
			=
			\eps_K|_B\otimes\eps_K|_B.
		\end{equation}
	\end{proposition}
	\begin{proof}
		Suppose that $\Gamma:A_\sigma\to K^\sigma$ is an algebra map. By
		\eqref{eq:deformed-Gamma}, for $b,c\in B$ we have
		\[
		b_1c_1\sigma^{-1}(b_2,c_2)
		=b\star_\sigma c
		=\sigma(b_1,c_1)b_2c_2\sigma^{-1}(b_3,c_3).
		\]
		Applying $\eps_K$ to both sides gives
		$\sigma^{-1}(b,c)=\eps_K(b)\eps_K(c)$. Since $B$ is a right
		coideal, the convolution identity then gives
		\[
		\eps_K(b)\eps_K(c)
		=\sigma^{-1}(b_1,c_1)\sigma(b_2,c_2)
		=\eps_K(b_1)\eps_K(c_1)\sigma(b_2,c_2)
		=\sigma(b,c),
		\]
		proving \eqref{eq:Gamma-hopf-target}.

		Conversely, assume \eqref{eq:Gamma-hopf-target}. Since $B$ is a
		right coideal, for $b,c\in B$ we have
		\[
		b\star_\sigma c
		=\sigma(b_1,c_1)b_2c_2\sigma^{-1}(b_3,c_3)
		=\eps_K(b_1)\eps_K(c_1)b_2c_2\sigma^{-1}(b_3,c_3)
		=b_1c_1\sigma^{-1}(b_2,c_2)
		=b\cdot_\sigma c.
		\]
		Together with \eqref{eq:deformed-Gamma}, this proves that
		$\Gamma:A_\sigma\to K^\sigma$ is multiplicative.
	\end{proof}
	
	Suppose that $\Gamma$ is injective. Since $\Gamma$ is a morphism in
	$\Alg(\Comod(K))$, the coaction is
	\[
	\rho=(\Gamma^{-1}\otimes\id)\Delta_K\Gamma,
	\qquad
	\rho(a)=a_0\otimes a_1,
	\]
	with $\Gamma^{-1}:B\to A$ the inverse of $\Gamma:A\to B$, and this
	is the unique coaction for which
	$\Gamma:(A,\rho)\to(K,\Delta_K)$ is a morphism in
	$\Alg(\Comod(K))$. Conversely, any injective unital algebra map
	$\Gamma:A\to K$ whose image is a right coideal subalgebra induces a
	right $K$-comodule algebra structure on $A$ by this formula.
	
	Since $I=0$, Lemma~\ref{lem:coideal-deformation} gives
	$\Gamma:A_\sigma\hookrightarrow F_\sigma(K)$.
	
	\begin{remark}\label{rem:hopf-target-closure}
		A right coideal subalgebra of $K$ need not remain a subalgebra
		of $K^\sigma$ after cocycle deformation. Over a field of characteristic
		$\neq2$, let $K$ be the Hopf algebra generated by $g,x$ with
		\[
		g^4=1,\quad x^2=0,\quad xg=-gx,\qquad
		\Delta(g)=g\otimes g,\quad
		\Delta(x)=x\otimes1+g\otimes x.
		\]
		The set $\{g^ix^e\mid 0\leq i<4,\ e\in\{0,1\}\}$ is a basis of $K$. Define
		\[
		\sigma(g^i,g^j)=1,\qquad
		\sigma(g^ix,g^j)=\sigma(g^i,g^jx)=0,\qquad
		\sigma(g^ix,g^jx)=(-1)^j.
		\]
		This is the specialization of the cocycle in
		\cite[Example~5.1]{GM} to the present parameters and basis. Its convolution inverse satisfies
		$\sigma^{-1}(g^ix,g^jx)=-(-1)^j$.
		For $y=g^{-1}x$, the subspace $B=k1+ky$ is a right coideal
		subalgebra, since $y^2=0$ and
		$\Delta(y)=y\otimes g^{-1}+1\otimes y$. The two products are
		\[
		y\cdot_\sigma y=1\in B,
		\qquad
		y\star_\sigma y=1-g^2\notin B.
		\]
		Thus the inclusion $B_\sigma\to K^\sigma$ induced by $\Gamma$ is
		not an algebra map.
	\end{remark}

	Recall that $(A,\rho)$ is a right $K$-comodule algebra.
	For $t\in\Alg(A,k)$, define the character-evaluation map
	\[
	j_t=(t\otimes\id_K)\rho:A\to K.
	\]
	In Remark~\ref{rem:hopf-target-closure}, the characters
	$t_\pm\in\Alg(B_\sigma,k)$ given by $t_\pm(y)=\pm1$
	yield injective algebra maps
	\[
	j_{t_\pm}:B_\sigma\longrightarrow K^\sigma,
	\qquad
	j_{t_\pm}(y)=y\pm g^{-1}.
	\]

	Related uses of characters appear in the study of
	covariantized algebras by Kolb--Stokman
	\cite[Section~2.6]{KoSt09} and of finite-dimensional
	Hopf algebra actions on finitely generated commutative
	domains by Etingof--Walton
	\cite[Lemma~3.2 and equation~(3.3)]{EW14}.

	\begin{proposition}[Character evaluation]
		\label{pro:character}
		Let $L$ be a Hopf algebra and $(C,\rho_C)$ a right
		$L$-comodule algebra. The assignment
		\[
		\Alg(C,k)\xrightarrow{\sim}
		\Hom_{\Alg(\Comod(L))}
		\bigl((C,\rho_C),(L,\Delta_L)\bigr),
		\qquad
		t\longmapsto j_t=(t\otimes\id)\rho_C,
		\]
		is a bijection, with inverse $j\mapsto\eps_L\circ j$.
	\end{proposition}
	
	\begin{proof}
		For $t\in\Alg(C,k)$, multiplicativity and
		coassociativity of $\rho_C$ give
		\[
		j_t(ab)=j_t(a)j_t(b),\qquad j_t(1)=1,
		\qquad
		\Delta_L\circ j_t=(j_t\otimes\id)\rho_C.
		\]
		Conversely, if $j$ is a morphism in
		$\Alg(\Comod(L))$, then $\eps_L\circ j\in\Alg(C,k)$,
		and colinearity gives
		\[
		j=((\eps_L\circ j)\otimes\id)\rho_C.
		\]
		Hence the two assignments are inverse.
	\end{proof}
	
	\begin{proposition}\label{prop:character-kernel}
		Let $L$ be a Hopf algebra, $(C,\rho_C)$ a right
		$L$-comodule algebra, and $t\in\Alg(C,k)$.
		Then $\ker j_t$ is the largest right
		$L$-subcomodule of $C$ contained in $\ker t$.
		Moreover, $\operatorname{im}j_t$ is a right coideal
		subalgebra of $L$, and
		\[
		C/\ker j_t\cong\operatorname{im}j_t
		\]
		as right $L$-comodule algebras.
	\end{proposition}
	
	\begin{proof}
		By Proposition~\ref{pro:character}, $j_t$ is a
		unital colinear algebra map. Thus its kernel is
		a right $L$-subcomodule ideal, its image is a
		right coideal subalgebra, and the stated
		isomorphism follows.
		
		Since $\eps_Lj_t=t$, one has
		$\ker j_t\subseteq\ker t$.
		If $U\subseteq\ker t$ is a right $L$-subcomodule,
		then
		\[
		j_t(U)=(t\otimes\id)\rho_C(U)=0.
		\]
		Hence $U\subseteq\ker j_t$.
	\end{proof}
	
	\begin{corollary}\label{cor:deformed-character-evaluation}
		The assignment
		\[
		\Alg(A_\sigma,k)
		\xrightarrow{\sim}
		\Hom_{\Alg(\Comod(K^\sigma))}
		\bigl((A_\sigma,\rho),(K^\sigma,\Delta_K)\bigr),
		\qquad
		t\longmapsto j_t=(t\otimes\id)\rho,
		\]
		is a bijection.
	\end{corollary}
	
	\begin{proof}
		Apply Proposition~\ref{pro:character} to
		the right $K^\sigma$-comodule algebra $A_\sigma$.
	\end{proof}

	\begin{corollary}\label{cor:deformed-characters}
		Let $t:A\to k$ be a linear map. Then
		$t\in\Alg(A_\sigma,k)$ if and only if
		$t(1)=1$ and
		\begin{equation}\label{eq:deformed-character-condition}
			t(a_0b_0)\sigma^{-1}(a_1,b_1)=t(a)t(b),\qquad a,b\in A.
		\end{equation}
		Moreover, the following conditions are equivalent:
		\begin{enumerate}[label=(\roman*)]
			\item $\Gamma:A_\sigma\to K^\sigma$ is an algebra map;
			\item $\eps_K\circ\Gamma\in\Alg(A_\sigma,k)$;
			\item $\sigma|_{B\otimes B}=\eps_K|_B\otimes\eps_K|_B$.
		\end{enumerate}
		
		Finally, let $r\in\Alg(A,k)$ and set
		$j_r=(r\otimes\id)\rho$ and
		$B_r=\operatorname{im}j_r$. Then
		\begin{equation}\label{eq:character-survival}
			r\in\Alg(A_\sigma,k)
			\quad\Longleftrightarrow\quad
			\sigma|_{B_r\otimes B_r}
			=
			\eps_K|_{B_r}\otimes\eps_K|_{B_r}.
		\end{equation}
	\end{corollary}
	
	\begin{proof}
		The first assertion is the multiplicativity
		condition for $t$ on $A_\sigma$.
		
		Since $\Gamma$ is right colinear, $\Gamma=((\eps_K\Gamma)\otimes\id)\rho$.
		The second assertion follows from
		Proposition~\ref{pro:character} and
		Proposition~\ref{prop:Gamma-hopf-target}.
		
		For $r\in\Alg(A,k)$,
		Proposition~\ref{pro:character} gives a morphism
		$j_r:A\to K$ in $\Alg(\Comod(K))$ with image
		$B_r$ and $\eps_Kj_r=r$.
		Applying the second assertion to $j_r$ yields
		\eqref{eq:character-survival}.
	\end{proof}
	
	\begin{remark}\label{rem:character-kernel}
		The character-evaluation map can have a nonzero kernel.
		Take $K=k[z,z^{-1}]$, $A=k[z]\subseteq K$,
		and the trivial cocycle. For the character
		$t(z)=0$, the coaction $\rho(z)=z\otimes z$
		gives $j_t(z)=0$ and $\operatorname{im}j_t=k1$.
	\end{remark}
	
	\begin{proposition}\label{prop:stable-ideals}
		Suppose that $A$ has no nonzero proper
		two-sided ideal that is a right $K$-subcomodule.
		Then every character-evaluation map of
		$A_\sigma$ is injective.
	\end{proposition}
	
	\begin{proof}
		Let $j_t$ be a character-evaluation map of $A_\sigma$. By Corollary~\ref{cor:deformed-character-evaluation},
		$j_t$ is a unital right $K^\sigma$-colinear
		algebra map. Thus $\ker j_t$ is a proper
		two-sided ideal of $A_\sigma$ and a right
		$K$-subcomodule of $A$.
		By Lemma~\ref{lem:deformation-ideals},
		it is also an ideal of $A$.
		Hence $\ker j_t=0$.
	\end{proof}
	
	\begin{theorem}\label{thm:coideal-character-injectivity}
		Let $B\subseteq K$ be a unital right coideal subalgebra,
		and  $j_t$  the character-evaluation map
		associated with $t\in\Alg(B_\sigma,k)$.
		The following conditions are equivalent:
		\begin{enumerate}
			\item $j_t$ is injective;
			\item $j_t|_{B\cap K_0}$ is injective;
			\item $t$ extends to a convolution-invertible
			functional $f\in K^*$;
			\item $j_t$ extends to a linear automorphism $U:K\to K$
			satisfying $\Delta_K\circ U=(U\otimes\id_K)\Delta_K$.
		\end{enumerate}
		If $K$ is pointed, these conditions are equivalent to
		\[
		t(g)\neq0
		\qquad\text{for all }g\in G(K)\cap B,
		\]
		where $G(K)$ denotes the group-like elements of $K$.
	\end{theorem}
	
	\begin{proof}
		Since the socle of the regular right $K$-comodule
		is $K_0$, the socle of $B$ is $B\cap K_0$.
		Every nonzero subcomodule contains a simple
		subcomodule. Since $\ker j_t$ is a subcomodule,
		conditions (1) and (2) are equivalent.
		
		Every right-colinear endomorphism $U$ of
		$(K,\Delta_K)$ has the form
		\[
		U=U_f,\qquad U_f(x)=f(x_1)x_2,
		\qquad f=\eps_KU.
		\]
		Indeed, this follows by applying
		$\eps_K\otimes\id$ to
		$\Delta_KU=(U\otimes\id)\Delta_K$.
		Moreover, $U_f\circ U_g=U_{g*f}$.
		Thus a convolution-invertible extension $f$
		of $t$ gives an automorphism $U_f$ with
		inverse $U_{f^{-1}}$ and $U_f|_B=j_t$.
		Conversely, if an automorphism $U$ extends
		$j_t$, then $f=\eps_KU$ and
		$g=\eps_KU^{-1}$ satisfy
		$f*g=g*f=\eps_K$, with $f|_B=t$.
		Hence (3) and (4) are equivalent, and each
		implies (1).
		
		Assume (2). Write $K_0=\bigoplus_D D$,
		where $D$ ranges over the simple subcoalgebras.
		Since distinct $D$ have nonisomorphic simple
		right-comodule constituents,
		\[
		B\cap K_0=\bigoplus_D(B\cap D),
		\]
		and $j_t(B\cap D)\subseteq D$.
		Each $D$ is a finite-dimensional semisimple
		right comodule, so the injective comodule map
		\[
		j_t|_{B\cap D}:B\cap D\longrightarrow D
		\]
		extends to a right-comodule automorphism of $D$: the source and
		image have isomorphic complements in the finite-dimensional
		semisimple comodule $D$.
		Taking the direct sum gives a right-comodule
		automorphism $U_0$ of $K_0$.
		
		Put $f_0=\eps_KU_0$. By the preceding
		endomorphism calculation, $f_0$ is
		convolution-invertible on $K_0$, and
		$f_0=t$ on $B\cap K_0$ since
		$\eps_Kj_t=t$. Thus $t$ on $B$ and $f_0$
		on $K_0$ define a functional on $B+K_0$.
		Extend it linearly to $f\in K^*$.
		By \cite[Lemma~14]{Takeuchi71}, a functional
		on a coalgebra is convolution-invertible
		if and only if its restriction to the
		coradical is convolution-invertible.
		Hence $f$ satisfies (3).
		
		If $K$ is pointed, then $K_0=\oplus_{g\in G(K)}kg$. 
		Since these one-dimensional summands are
		pairwise nonisomorphic simple comodules, $B\cap K_0=\operatorname{span}_k(G(K)\cap B).$
		As $j_t(g)=t(g)g$, condition (2) is
		equivalent to the stated nonvanishing condition.
	\end{proof}

	\begin{lemma}[Injectivity for Hopf quotients]
		\label{lem:quotient-character-injectivity}
		Let $\pi:K\to Q$ be a Hopf algebra map,
		and let $(Q,\rho)\in\Alg(\Comod(K))$ satisfy
		\begin{equation}\label{eq:quotient-coaction}
			(\id_Q\otimes\pi)\rho=\Delta_Q.
		\end{equation}
		Then every character-evaluation map of $(Q,\rho)_\sigma$ is injective.
	\end{lemma}
	
	\begin{proof}
		Let $J\subseteq Q$ be a nonzero two-sided ideal
		that is a right $K$-subcomodule. Then $\Delta_Q(J)\subseteq J\otimes Q.$
		Hence $\eps_Q(J)\neq0$, since otherwise
		$J=(\eps_Q\otimes\id_Q)\Delta_Q(J)=0$.
		Choose $q\in J$ with $\eps_Q(q)\neq0$.
		Since $J$ is a right ideal,
		\[
		\eps_Q(q)1=q_1S_Q(q_2)\in J,
		\]
		so $J=Q$. The assertion follows from
		Proposition~\ref{prop:stable-ideals}.
	\end{proof}
	\begin{remark}\label{rem:quotient-coaction-section}
		The condition \eqref{eq:quotient-coaction}
		gives an algebra section
		\[
		\Gamma_\rho
		=(\eps_Q\otimes\id_K)\rho:Q\to K
		\]
		satisfying
		\[
		\pi\Gamma_\rho=\id_Q,\qquad
		\Delta_K\Gamma_\rho
		=(\Gamma_\rho\otimes\id_K)\rho.
		\]
		Thus $\Gamma_\rho$ is a morphism in
		$\Alg(\Comod(K))$ and induces an isomorphism
		\[
		(Q,\rho)\cong
		(\Gamma_\rho(Q),
		\Delta_K|_{\Gamma_\rho(Q)}).
		\]
		Moreover,
		\[
		\rho=(\pi\otimes\id_K)
		\Delta_K\Gamma_\rho.
		\]
	\end{remark}
	
	\begin{lemma}[Comodule-algebra isomorphisms]
		\label{lem:quotient-coaction-morphisms}
		Let $\pi:K\to Q$ be a Hopf algebra map,
		and let $(Q,\rho),(Q,\rho')\in\Alg(\Comod(K))$
		each satisfy \eqref{eq:quotient-coaction}.
		For $r\in\Alg(Q,k)$, set
		\[
		U_r=(r\otimes\id_Q)\Delta_Q,\qquad
		\rho_r=(U_r\otimes\id_K)\rho U_{r^{-1}},
		\qquad r^{-1}=r\circ S_Q.
		\]
		Then $\rho_r$ is a right $K$-comodule algebra
		structure satisfying \eqref{eq:quotient-coaction}.
		
		Let $A=(Q,\rho)$ and $A'=(Q,\rho')$.
		The morphisms $A_\sigma\to A'_\sigma$ in
		$\Alg(\Comod(K^\sigma))$ are precisely
		the maps $U_r$ with $\rho'=\rho_r$.
		Every such morphism is an isomorphism,
		with $r=\eps_Q\circ U_r$ uniquely determined.
	\end{lemma}
	\begin{proof}
		Since $r\in\Alg(Q,k)$, its convolution inverse
		$r^{-1}=r\circ S_Q$ is also an algebra character.
		For $a,b\in Q$,
		\[
		U_r(ab)
		=r(a_{(1)}b_{(1)})a_{(2)}b_{(2)}
		=U_r(a)U_r(b),
		\qquad U_r(1)=1.
		\]
		Moreover, for $r,s\in\Alg(Q,k)$,
		\[
		U_rU_s=U_{s*r},\qquad
		U_r^{-1}=U_{r^{-1}},
		\qquad
		\Delta_QU_r=(U_r\otimes\id_Q)\Delta_Q.
		\]
		Thus $U_r$ is an algebra automorphism of $Q$.
		
		Since $\rho$ and $U_r^{\pm1}$ are algebra maps,
		so is $\rho_r$. Furthermore,
		\[
		(\id_Q\otimes\eps_K)\rho_r
		=U_r(\id_Q\otimes\eps_K)\rho U_r^{-1}
		=\id_Q,
		\]
		and
		\begin{align*}
			(\rho_r\otimes\id_K)\rho_r
			&=(U_r\otimes\id_K\otimes\id_K)
			(\rho\otimes\id_K)\rho U_r^{-1}\\
			&=(U_r\otimes\id_K\otimes\id_K)
			(\id_Q\otimes\Delta_K)\rho U_r^{-1}\\
			&=(\id_Q\otimes\Delta_K)\rho_r.
		\end{align*}
		Hence $\rho_r$ is a right $K$-comodule
		algebra structure. It also satisfies
		\[
		(\id_Q\otimes\pi)\rho_r
		=(U_r\otimes\id_Q)\Delta_QU_r^{-1}
		=\Delta_Q.
		\]
		
		The cocycle deformation equivalence
		$F_\sigma$ and its inverse preserve the
		underlying linear maps. Consequently,
		\[
		\Hom_{\Alg(\Comod(K^\sigma))}
		(A_\sigma,A'_\sigma)
		=
		\Hom_{\Alg(\Comod(K))}(A,A')
		\]
		as sets of linear maps.
		
		Let $T:A\to A'$ be a morphism in
		$\Alg(\Comod(K))$. Since $T$ is
		$K$-colinear and both coactions satisfy
		\eqref{eq:quotient-coaction},
		\begin{align*}
			\Delta_QT
			=(\id_Q\otimes\pi)\rho'T
			=(\id_Q\otimes\pi)(T\otimes\id_K)\rho
			=(T\otimes\id_Q)\Delta_Q.
		\end{align*}
		Applying $\eps_Q\otimes\id_Q$ gives
		\[
		T=(r\otimes\id_Q)\Delta_Q=U_r,
		\qquad r=\eps_Q\circ T.
		\]
		Since $T$ and $\eps_Q$ are unital algebra
		maps, $r\in\Alg(Q,k)$.
		
		The $K$-colinearity of $T=U_r$ is
		precisely
		\[
		\rho'U_r=(U_r\otimes\id_K)\rho,
		\]
		or equivalently,
		\[
		\rho'
		=(U_r\otimes\id_K)\rho U_{r^{-1}}
		=\rho_r.
		\]
		Conversely, if $\rho'=\rho_r$, then
		$U_r:A\to A'$ is a $K$-colinear algebra
		map. Its inverse $U_{r^{-1}}$ is also
		$K$-colinear, since
		\[
		\rho U_{r^{-1}}
		=(U_{r^{-1}}\otimes\id_K)\rho'.
		\]
		Thus $U_r$ is an isomorphism in
		$\Alg(\Comod(K))$, and hence also
		after cocycle deformation.
		
		Finally,
		\[
		\eps_Q\circ U_r=r,
		\]
		which proves the uniqueness of $r$.
	\end{proof}
	
	\begin{theorem}[Equality of images]
		\label{thm:character-evaluation-images}
		Under the assumptions of
		Lemma~\ref{lem:quotient-coaction-morphisms},
		let $t\in\Alg(A_\sigma,k)$ and
		$u\in\Alg(A'_\sigma,k)$, and put
		\[
		j_t^\rho=(t\otimes\id_K)\rho,\qquad
		j_u^{\rho'}=(u\otimes\id_K)\rho'.
		\]
		The following are equivalent:
		\begin{enumerate}[label=(\roman*)]
			\item
			$\operatorname{im}j_t^\rho
			=\operatorname{im}j_u^{\rho'}$.
			
			\item There is an isomorphism
			\[
			T:A_\sigma\xrightarrow{\sim}A'_\sigma
			\quad\text{in }\Alg(\Comod(K^\sigma))
			\]
			such that $u\circ T=t$.
			
			\item There is $r\in\Alg(Q,k)$ satisfying
			\begin{equation}\label{eq:character-evaluation-images}
				\rho'=\rho_r,\qquad
				u=r^{-1}*t.
			\end{equation}
		\end{enumerate}
		When these conditions hold, $T=U_r$, and
		both $T$ and $r$ are unique.
		
		Moreover, the opposite convolution group
		$\Alg(Q,k)^{\op}$ acts freely on the pairs
		$(\rho,t)$ satisfying
		\[
		(\id_Q\otimes\pi)\rho=\Delta_Q,\qquad
		t\in\Alg((Q,\rho)_\sigma,k),
		\]
		by
		\[
		r\cdot(\rho,t)=(\rho_r,r^{-1}*t).
		\]
		Its orbits are precisely the fibers of
		$(\rho,t)\mapsto\operatorname{im}j_t^\rho$.
		All convolutions use $\Delta_Q$.
	\end{theorem}
	\begin{proof}
		By Lemma~\ref{lem:quotient-character-injectivity},
		$j_t^\rho$ and $j_u^{\rho'}$ are injective.
		
		If (i) holds, then
		\[
		T=(j_u^{\rho'})^{-1}\circ j_t^\rho
		\]
		is an isomorphism in
		$\Alg(\Comod(K^\sigma))$, with the inverse
		taken on the common image. Applying $\eps_K$
		to $j_u^{\rho'}T=j_t^\rho$ gives $uT=t$.
		
		Conversely, (ii) gives
		\[
		j_u^{\rho'}T
		=(u\otimes\id_K)\rho'T
		=((uT)\otimes\id_K)\rho
		=j_t^\rho,
		\]
		proving (i). Injectivity of $j_u^{\rho'}$
		also gives the uniqueness of $T$.
		
		By Lemma~\ref{lem:quotient-coaction-morphisms},
		such an isomorphism is $T=U_r$ for a unique
		$r\in\Alg(Q,k)$ with $\rho'=\rho_r$.
		Moreover,
		\[
		uU_r=t
		\iff
		u=tU_{r^{-1}}=r^{-1}*t.
		\]
		Thus (ii) and (iii) are equivalent.
		
		The same lemma gives
		\[
		r^{-1}*t=tU_{r^{-1}}
		\in\Alg((Q,\rho_r)_\sigma,k),
		\]
		so the stated operation is well-defined.
		Since
		$\rho_rU_r=(U_r\otimes\id_K)\rho$,
		\begin{align*}
			j_{r^{-1}*t}^{\rho_r}U_r
			&=((tU_{r^{-1}})\otimes\id_K)
			\rho_rU_r\\
			&=(tU_{r^{-1}}U_r\otimes\id_K)\rho\\
			&=j_t^\rho.
		\end{align*}
		In particular,
		\[
		\operatorname{im}j_{r^{-1}*t}^{\rho_r}
		=\operatorname{im}j_t^\rho.
		\]
		
		The identities
		\[
		(\rho_s)_r=\rho_{s*r},\qquad
		r^{-1}*(s^{-1}*t)=(s*r)^{-1}*t
		\]
		give
		\[
		r\cdot(s\cdot(\rho,t))
		=(s*r)\cdot(\rho,t).
		\]
		Together with the identity character $\eps_Q$,
		these establish an action of
		$\Alg(Q,k)^{\op}$.
		
		If $r\cdot(\rho,t)=(\rho,t)$, the preceding
		intertwining identity gives
		\[
		j_t^\rho U_r=j_t^\rho.
		\]
		Injectivity of $j_t^\rho$ implies
		$U_r=\id_Q$, hence $r=\eps_Q$.
		Thus the action is free. Finally, the
		equivalence of (i) and (iii) shows that
		its orbits are precisely the fibers of
		$(\rho,t)\mapsto\operatorname{im}j_t^\rho$.
	\end{proof}

	\begin{corollary}
		\label{cor:character-evaluation-extension}
		Under the assumptions and notation of
		Lemma~\ref{lem:quotient-coaction-morphisms},
		the map $\Gamma_\rho=(\eps_Q\otimes\id_K)\rho$
		induces an injective morphism
		$A_\sigma\to F_\sigma(K)$.
		For $r\in\Alg(Q,k)$,
		\[
		E_r(x)=(r\circ\pi)(x_1)x_2
		\]
		defines an automorphism of $F_\sigma(K)$
		in $\Alg(\Comod(K^\sigma))$ satisfying
		\begin{equation}\label{eq:character-extension}
			E_r\Gamma_\rho=\Gamma_{\rho_r}U_r.
		\end{equation}
	\end{corollary}
	
	\begin{proof}
		By Remark~\ref{rem:quotient-coaction-section},
		$\Gamma_\rho$ is an injective morphism
		in $\Alg(\Comod(K))$. Applying $F_\sigma$
		gives an injective morphism
		\[
		\Gamma_\rho:A_\sigma
		\longrightarrow F_\sigma(K).
		\]
		
		The winding automorphism
		$U_{r\circ\pi}$ of $K$ is a morphism
		in $\Alg(\Comod(K))$. Applying $F_\sigma$
		gives $E_r$, an automorphism of
		$F_\sigma(K)$; see also
		\cite[Corollary~3.1.4]{Sch04}.
		
		Finally,
		\begin{align*}
			E_r\Gamma_\rho
			=(r\circ\pi\otimes\id_K)\Delta_K\Gamma_\rho
			=(r\otimes\id_K)\rho
			=\Gamma_{\rho_r}U_r.
		\end{align*}
		
	\end{proof}

	\begin{proposition}\label{prop:group-algebras}
		Let $G$ be a group and $K=kG$.
		\begin{enumerate}
			\item Let $\alpha:G\times G\to k^\times$ be a
			normalized group $2$-cocycle, and let $\sigma$
			be its linear extension to $K\otimes K$.
			For a subgroup $H\leq G$, put $A=kH$,
			and let $\Gamma:A\hookrightarrow K$ be
			the inclusion. Then $K^\sigma=K$, and
			\[
			\Gamma:A_\sigma\to K^\sigma
			\text{ is an algebra map}
			\iff \alpha|_{H\times H}=1.
			\]
			
			\item For the trivial cocycle, let
			$M\subseteq G$ be a submonoid containing $1$,
			and  $B=kM$. For $t,u\in\Alg(B,k)$,
			write
			\[
			M_t=\{m\in M:t(m)\neq0\},\qquad
			M_u=\{m\in M:u(m)\neq0\}.
			\]
			Then $B$ is a right coideal subalgebra of $K$,
			and
			\[
			j_t(m)=t(m)m,\qquad
			\operatorname{im}j_t=kM_t.
			\]
			Moreover,
			\[
			\begin{aligned}
				j_t\text{ is injective}
				&\iff M_t=M,\qquad
				\operatorname{im}j_t=\operatorname{im}j_u
				&\iff M_t=M_u.
			\end{aligned}
			\]
		\end{enumerate}
	\end{proposition}
	
	\begin{proof}
		For $h,h'\in H$ and $g,g'\in G$,
		\[
		h\cdot_\sigma h'
		=\alpha(h,h')^{-1}hh',
		\qquad
		g\star_\sigma g'=gg'.
		\]
		This proves (1). For (2),
		$\Delta(m)=m\otimes m$ gives
		$j_t(m)=t(m)m$.
		The remaining assertions follow from
		the linear independence of $M$ in $kG$.
	\end{proof}

	\begin{remark}\label{rem:twisted-characters}
		Cocycle deformation need not preserve algebra characters.
		Let $A=K=k[x^{\pm1},y^{\pm1}]$ with group-like
		$x,y$, and choose $q\in k^\times$ with $q\neq1$.
		The cocycle
		\[
		\sigma(x^my^n,x^py^s)=q^{np}
		\]
		gives 
		$x\cdot_\sigma y=q\,y\cdot_\sigma x$;
		compare \cite[Example~3.7]{GKM12}.
		Since $x$ and $y$ remain invertible in $A_\sigma$,
		this relation implies that $A_\sigma$ has no
		characters, whereas $A$ has characters.
		
		In this example $K^\sigma=K$.
		For every $r\in\Alg(A,k)$, the undeformed map
		$j_r$ has image $K$, so the nontrivial restriction
		of $\sigma$ also rules out survival of $r$ by
		\eqref{eq:character-survival}.
	\end{remark}
	
	\begin{corollary}\label{cor:coideal-subalgebras}
		The assignment $B\mapsto B_\sigma$ is a bijection
		from the unital right coideal subalgebras of $K$
		to the unital right $K^\sigma$-comodule subalgebras
		of $F_\sigma(K)$, keeping each underlying subspace
		unchanged. It preserves and reflects inclusions,
		and hence preserves intersections and joins,
		where a join is the subalgebra generated by a union.
	\end{corollary}
	
	\begin{proof}
		Apply $F_{\sigma^{-1}}$ to a right
		$K^\sigma$-comodule subalgebra
		$C\subseteq F_\sigma(K)$. For $x,y\in C$,
		the recovered product is
		\[
		x_1\cdot_\sigma y_1\,\sigma(x_2,y_2)
		=
		x_1y_1\sigma^{-1}(x_2,y_2)\sigma(x_3,y_3)
		=
		xy.
		\]
		Thus $C$ is a right coideal subalgebra of $K$.
		This construction is inverse to $B\mapsto B_\sigma$.
		Since both assignments preserve inclusions,
		they also preserve intersections and joins.
	\end{proof}
	
	\begin{corollary}[Subalgebras generated by subcomodules]
		\label{cor:coideal-generators}
		Let $V\subseteq K$ be a right subcomodule.
		Then the subalgebra generated by $V$ in $K$
		and the subalgebra generated by $V$ in
		$F_\sigma(K)$ have the same underlying subspace.
	\end{corollary}
	
	\begin{proof}
		The two generated algebras are the least
		comodule subalgebras containing $V$, so the
		assertion follows from
		Corollary~\ref{cor:coideal-subalgebras}.
	\end{proof}
	
	Compare \cite[Lemmas~3.1--3.2 and Theorem~3.3]{GKM12},
	which treat generators and defining relations under
	cocycle twisting for left comodule algebras.

	For the regular right $K$-Galois object $K$, the right
	Miyashita--Ulbrich action is the right adjoint action.
	We next examine the stability of comodule subalgebras
	under the Miyashita--Ulbrich action after cocycle deformation.

	\begin{lemma}\label{lem:MU-stability}
		Let $B\subseteq K$ be a unital right coideal subalgebra.
		Then $B_\sigma\subseteq F_\sigma(K)$ is stable under
		the right Miyashita--Ulbrich action of $K^\sigma$
		if and only if
		\[
		S_K(h_1)bh_2\in B,
		\qquad b\in B,\ h\in K.
		\]
	\end{lemma}
	
	\begin{proof}
		For $\beta=\sigma,\sigma^{-1}$, define
		\[
		T_\beta(x\otimes h)
		=x_1\otimes h_1\beta(x_2,h_2).
		\]
		Write $\leftharpoonup_\sigma$ for the right
		Miyashita--Ulbrich action of $K^\sigma$ on
		$F_\sigma(K)$, and put
		\[
		c(x\otimes h)
		=h_1\otimes S_K(h_2)xh_3,\qquad
		c_\sigma(x\otimes h)
		=h_1\otimes(x\leftharpoonup_\sigma h_2).
		\]
		
		The convolution identities give
		$T_\sigma^{-1}=T_{\sigma^{-1}}$.
		Put $d=T_\sigma cT_{\sigma^{-1}}$.
		A direct expansion gives
		\[
		d(x\otimes h)=h_1\otimes a(x,h_2),
		\qquad a=(\eps_K\otimes\id)d.
		\]
		Let $m$ and $m_\sigma=mT_{\sigma^{-1}}$
		denote the multiplications of $K$ and
		$F_\sigma(K)$, respectively.
		Since $mc=m$,
		\[
		m_\sigma d
		=mT_{\sigma^{-1}}T_\sigma
		cT_{\sigma^{-1}}
		=mT_{\sigma^{-1}}
		=m_\sigma.
		\]
		Thus $x\cdot_\sigma y=y_1\cdot_\sigma a(x,y_2)$.
		Recall that $\can_R^{-1}(1\otimes h)=h^{[1]}\otimes h^{[2]}.$
		
		Since
		\[
		h^{[1]}\cdot_\sigma h^{[2]}_1
		\otimes h^{[2]}_2=1\otimes h,
		\]
		we obtain
		\[
		h^{[1]}\cdot_\sigma x\cdot_\sigma h^{[2]}
		=h^{[1]}\cdot_\sigma h^{[2]}_1
		\cdot_\sigma a(x,h^{[2]}_2)
		=a(x,h).
		\]
		Hence
		\[
		a(x,h)=x\leftharpoonup_\sigma h,
		\qquad
		c_\sigma=T_\sigma cT_{\sigma^{-1}}.
		\]
		
		Since $B$ is a right coideal,
		both $T_\sigma$ and $T_{\sigma^{-1}}$
		preserve $B\otimes K$ and $K\otimes B$.
		Consequently,
		\[
		c_\sigma(B\otimes K)\subseteq K\otimes B
		\iff
		c(B\otimes K)\subseteq K\otimes B.
		\]
		The first inclusion is equivalent to the
		stability of $B_\sigma$ under
		$\leftharpoonup_\sigma$, while the second
		is equivalent to
		\[
		S_K(h_1)bh_2\in B,
		\qquad b\in B,\ h\in K.
		\]
		Indeed, stability implies the respective
		inclusion, and the converse follows
		by applying $\eps_K\otimes\id$.
	\end{proof}

	\begin{corollary}\label{cor:group-MU-stability}
		Let $G$ be a group with a normalized group
		$2$-cocycle $\alpha:G\times G\to k^\times$,
		and let $\sigma$ be the corresponding Hopf
		$2$-cocycle on $K=kG$.
		For any submonoid $M\subseteq G$ containing $1$,
		set $B=kM$. Then $B_\sigma\subseteq F_\sigma(K)$
		is stable under the right Miyashita--Ulbrich action
		if and only if
		\[
		g^{-1}Mg=M,\qquad g\in G.
		\]
		In particular, for a subgroup $H\leq G$,
		$(kH)_\sigma$ is stable if and only if
		$H$ is normal in $G$.
	\end{corollary}
	
	\begin{proof}
		Since $K^\sigma=K$ by
		Proposition~\ref{prop:group-algebras},
		the right canonical map gives
		\[
		\can_R^{-1}(1\otimes g)
		=\alpha(g^{-1},g)\,g^{-1}\otimes g.
		\]
		Thus, for $m,g\in G$, the right
		Miyashita--Ulbrich action is
		\[
		m\leftharpoonup_\sigma g
		=
		\frac{\alpha(g^{-1},g)}
		{\alpha(g^{-1},m)\alpha(g^{-1}m,g)}
		\,g^{-1}mg
		=
		\frac{\alpha(g,g^{-1}mg)}{\alpha(m,g)}
		\,g^{-1}mg,
		\]
		where the second equality follows from the
		cocycle identity on $(g^{-1},m,g)$ and
		$(g,g^{-1},mg)$, together with
		$\alpha(g,g^{-1})=\alpha(g^{-1},g)$.
		For the trivial cocycle, this reduces to
		$m\leftharpoonup g=g^{-1}mg$.
		
		Since the scalar coefficients are nonzero,
		stability is equivalent to
		$g^{-1}Mg\subseteq M$ for all $g\in G$.
		Applying the same inclusion to $g^{-1}$
		gives equality.
	\end{proof}

	\section{\texorpdfstring{$\imath$}{i}-Hopf algebras and Drinfeld doubles}
	\label{sec:ihopf}
	
	Fix a Hopf algebra $H$ with bijective antipode $S$
	and a Hopf pairing $\phi:H\otimes H\to k$. Write
	\[
	\EndHopf(H)=\Hom_{\mathrm{Hopf}}(H,H),\qquad
	\AutHopf(H)
	=\{f\in\EndHopf(H)\mid f\text{ is bijective}\},
	\]
	where $\Hom_{\mathrm{Hopf}}$ denotes Hopf algebra maps.
	Fix $\tau\in\EndHopf(H)$ and put
	$K=H\otimes H^{\cop}$. Consider the map
	\[
	\Gamma_\tau:H\to K,\qquad
	\Gamma_\tau(h)=\tau(h_2)\otimes h_1.
	\]
	We first recall the pairing constructions that
	define the Hopf $2$-cocycle on $K$ and the
	associated $\imath$-Hopf algebras.

	\subsection{Skew-pairings and generalized quantum doubles}
	For this construction from a skew-pairing by Hopf $2$-cocycle deformation, see \cite{DT94} and also \cite{KS05}.
	\begin{definition}
		Let $U$ and $V$ be Hopf algebras, with bijective antipode on
		$V$. A \emph{skew-pairing}
		$\lambda:U\otimes V\to k$ satisfies
		\begin{align}
			\lambda(uv,a)
			&=
			\lambda(u,a_1)\lambda(v,a_2),
			\label{eq:skew-first}\\
			\lambda(u,ab)
			&=
			\lambda(u_2,a)\lambda(u_1,b),
			\label{eq:skew-second}\\
			\lambda(u,1)&=\eps(u),
			\qquad
			\lambda(1,a)=\eps(a).
		\end{align}
	\end{definition}
	
	Its convolution inverse is
	\[
	\lambda^{-1}(u,a)
	=
	\lambda(S_U(u),a)
	=
	\lambda(u,S_V^{-1}(a)).
	\]
	
	\begin{proposition}[Generalized quantum doubles]
		\label{prop:generalized-double}
		The linear map
		\begin{equation}\label{eq:pairing-cocycle}
			\sigma(u\otimes a,v\otimes b)
			=
			\eps(u)\lambda(v,a)\eps(b)
		\end{equation}
		is a Hopf $2$-cocycle on $U\otimes V$.
		The Hopf $2$-cocycle deformation
		$(U\otimes V)^\sigma$ is the generalized quantum double
		associated with $U,V$, and $\lambda$, with multiplication
		\begin{equation}\label{eq:general-double}
			(u\otimes a)\star_\sigma(v\otimes b)
			=
			u\,\lambda(v_1,a_1)v_2
			\otimes
			a_2\,\lambda^{-1}(v_3,a_3)b.
		\end{equation}
	\end{proposition}

	\subsection{Hopf pairings and \texorpdfstring{$\imath$}{i}-Hopf algebras}
	
	Following \cite[Section~2.1]{CLPRW}, the \emph{Hopf pairing}
	$\phi:H\otimes H\to k$ satisfies
	\begin{align}
		\phi(a,1)&=\eps(a),&
		\phi(1,b)&=\eps(b),\notag\\
		\phi(a,bb')&=\phi(a_1,b)\phi(a_2,b'),&
		\phi(aa',b)&=\phi(a,b_1)\phi(a',b_2),
		\label{eq:pairing}\\
		\phi(a,S(b))&=\phi(S(a),b). \notag
	\end{align}
	
	The $\imath$-Hopf algebra associated with $(H,\phi)$, denoted
	$H^\imath$, is the vector space $H$ with multiplication
	\[
	a*b=\phi(b_2,a_1)a_2b_1;
	\]
	see \cite[Definition~2.1]{CLPRW}. It is an associative algebra with
	unit $1$ by \cite[Proposition~2.2]{CLPRW}.
	
	By \cite[Lemma~2.5]{CLPRW},
	$\phi\circ(\tau\otimes\id)$ is again a Hopf pairing. Following
	\cite[Section~2.4]{CLPRW}, we denote the corresponding
	$\imath$-Hopf algebra by $H^\imath_\tau$. Its multiplication is
	\begin{equation}\label{eq:iproduct}
		a*_\tau b
		=
		\phi(\tau(b_2),a_1)a_2b_1.
	\end{equation}
	
	\begin{remark}
		In general, $H^\imath$ is only an associative algebra. The
		$\imath$-Hopf algebra of diagonal type below carries an
		additional Hopf algebra structure.
	\end{remark}

	\subsection{\texorpdfstring{$\imath$}{i}-Hopf algebras of diagonal type}
	
	Following \cite[Section~2.2]{CLPRW}, endow the tensor product Hopf
	algebra $H\otimes H$ with the Hopf pairing
	\[
	\phi_\sharp(c\otimes d,a\otimes b)
	=
	\phi(a,d)\phi(c,b).
	\]
	The associated $\imath$-Hopf algebra
	$(H\otimes H)^\imath$ is called the
	\emph{$\imath$-Hopf algebra of diagonal type}. Its multiplication is
	\begin{equation}\label{eq:Jproduct}
		(u\otimes a)*(v\otimes b)
		=
		\phi(u_1,b_2)\phi(v_2,a_1)
		u_2v_1\otimes a_2b_1.
	\end{equation}
	
	By \cite[Proposition~2.7]{CLPRW},
	$(H\otimes H)^\imath$ is a Hopf algebra with
	\begin{align}
		\Delta^\imath(u\otimes a)
		&=
		\phi(u_2,a_2)
		(u_1\otimes a_3)\otimes(u_3\otimes a_1),
		\label{eq:Jdelta}\\
		\eps^\imath(u\otimes a)
		&=
		\phi(u,S^{-1}(a)),&
		S^\imath(u\otimes a)
		&=
		S(u)\otimes S^{-1}(a).
		\label{eq:Jcounit}
	\end{align}

	\subsection{The Drinfeld double}
	Set $K=H\otimes H^{\cop}$. The algebra $H^{\cop}$ has the
	multiplication of $H$, the opposite coproduct, and antipode $S^{-1}$.
	Unless indicated otherwise, Sweedler indices on elements of $H$
	refer to $\Delta_H$. Thus the coproduct of $K$ is
	\begin{equation}\label{eq:Kdelta}
		\Delta_K(u\otimes a)
		=
		(u_1\otimes a_2)\otimes(u_2\otimes a_1).
	\end{equation}
	
	Take $U=H$, $V=H^{\cop}$, and define
	\[
	\lambda(v,a)=\phi(v,S^{-1}(a)).
	\]
	In terms of the original coproduct of $H$,
	\begin{align*}
		\lambda(vw,a)
		&=
		\phi(v,S^{-1}(a_2))\phi(w,S^{-1}(a_1)),\\
		\lambda(v,ab)
		&=
		\phi(v_1,S^{-1}(b))\phi(v_2,S^{-1}(a)).
	\end{align*}
	Thus $\lambda$ is a skew-pairing, with
	\[
	\lambda^{-1}(v,a)
	=
	\phi(S(v),S^{-1}(a))
	=
	\phi(v,a).
	\]
	
	The associated Hopf $2$-cocycle on $K$ is
	\begin{align}
		\sigma(u\otimes a,v\otimes b)
		&=
		\eps(u)\phi(v,S^{-1}(a))\eps(b),
		\label{eq:sigma}\\
		\sigma^{-1}(u\otimes a,v\otimes b)
		&=
		\eps(u)\phi(v,a)\eps(b).
		\label{eq:sigma-inverse}
	\end{align}
	By Proposition~\ref{prop:generalized-double}, the multiplication of
	$K^\sigma$ is
	\begin{equation}\label{eq:Dproduct}
		(u\otimes a)\star_\sigma(v\otimes b)
		=
		\phi(v_1,S^{-1}(a_3))\phi(v_3,a_1)
		uv_2\otimes a_2b.
	\end{equation}
	
	\begin{proposition}\label{prop:cocycle-double}
		The identity map on $H\otimes H$ defines a Hopf algebra
		isomorphism $K^\sigma\xrightarrow{\sim}D(H)$,
		where $D(H)$ is the Drinfeld double of
		\cite[Section~2.3]{CLPRW}.
	\end{proposition}
	
	\begin{proof}
		The coalgebra \eqref{eq:Kdelta} agrees with that of $D(H)$.
		The multiplication \eqref{eq:Dproduct} satisfies relations
		{\rm(D1)}--{\rm(D3)} of \cite[Section~2.3]{CLPRW}. For the remaining
		mixed relation,
		\begin{align*}
			\phi(v_1,a_2)
			(1\otimes a_1)\star_\sigma(v_2\otimes1)
			&=
			\phi(v_1,a_4)\phi(v_2,S^{-1}(a_3))
			\phi(v_4,a_1)v_3\otimes a_2\\
			&=
			\phi(v_2,a_1)v_1\otimes a_2,
		\end{align*}
		where the last equality uses the Hopf-pairing identities and
		$\lambda^{-1}*\lambda=\eps\otimes\eps$. This is relation {\rm(D4)}.
	\end{proof}
	
	We henceforth identify $K^\sigma=D(H)$.  By \cite[Proposition~2.8]{CLPRW}, there is a Hopf algebra
	isomorphism
	\begin{equation}\label{eq:PhiD}
		\Phi^{\mathrm D}:D(H)\xrightarrow{\;\sim\;}
		(H\otimes H)^\imath,
		\qquad
		\Phi^{\mathrm D}(u\otimes a)
		=
		\phi(u_1,a_2)u_2\otimes a_1,
	\end{equation}
	with inverse
	\begin{equation}\label{eq:PhiDinv}
		(\Phi^{\mathrm D})^{-1}(u\otimes a)
		=
		\phi(u_1,S^{-1}(a_2))u_2\otimes a_1.
	\end{equation}

	Finally, the right coaction
	$\Psi:H^\imath_\tau\to H^\imath_\tau\otimes(H\otimes H)^\imath$
	of \cite[Proposition~2.11]{CLPRW} is given by
	\begin{equation}\label{eq:Psi}
		\Psi(h)
		=
		\phi(\tau(h_4),h_2)h_3
		\otimes(\tau(h_5)\otimes h_1).
	\end{equation}
	In the next section, the Hopf Galois realization of
	$H^\imath_\tau$ in $F_\sigma(K)$ recovers \eqref{eq:Psi} and
	its right $((H\otimes H)^\imath)$-comodule algebra structure.
	
	\section{The Hopf Galois realization of \texorpdfstring{$\imath$}{i}-Hopf algebras}
	\label{sec:realization}

	In this section we apply Lemma~\ref{lem:coideal-deformation} to
	$\Gamma_\tau$. The cocycle deformation of $H$ is thereby
	identified with $H^\imath_\tau$, and we obtain an embedding $\Gamma_\tau:H^\imath_\tau\hookrightarrow F_\sigma(K)$
	of right $K^\sigma$-comodule algebras. Under the isomorphism
	$\Phi^{\mathrm D}:K^\sigma\to(H\otimes H)^\imath$ of
	\eqref{eq:PhiD}, the induced coaction agrees with the coaction
	$\Psi$ of \cite[Proposition~2.11]{CLPRW}.
	
	\begin{lemma}\label{lem:graph}
		The map
		\[
		\Gamma_\tau:H\longrightarrow K,
		\qquad
		\Gamma_\tau(h)=\tau(h_2)\otimes h_1,
		\]
		belongs to $\Alg(H,K)$, is injective, and has a right coideal
		subalgebra as its image. Via $\Gamma_\tau$, the induced right
		$K$-coaction $\delta=\delta_\tau:H\to H\otimes K$ is
		\begin{equation}\label{eq:delta}
			\delta(h)
			=
			h_2\otimes\bigl(\tau(h_3)\otimes h_1\bigr).
		\end{equation}
	\end{lemma}
	
	\begin{proof}
		Since $\tau$ is an algebra homomorphism,
		\[
		\Gamma_\tau(ab)
		=
		\tau(a_2)\tau(b_2)\otimes a_1b_1
		=
		\Gamma_\tau(a)\Gamma_\tau(b),
		\qquad
		\Gamma_\tau(1)=1.
		\]
		Moreover, $(\eps\otimes\id)\Gamma_\tau=\id_H$, so $\Gamma_\tau$ is injective. Using \eqref{eq:Kdelta},
		\[
		\Delta_K\Gamma_\tau(h)
		=
		(\tau(h_3)\otimes h_2)
		\otimes
		(\tau(h_4)\otimes h_1)
		=
		(\Gamma_\tau\otimes\id_K)\delta(h).
		\]
		Thus $\Delta_K(\Gamma_\tau(H))\subseteq\Gamma_\tau(H)\otimes K$.
		The restricted coproduct therefore induces a coaction on $H$,
		and its formula is \eqref{eq:delta}.
	\end{proof}
	
	\begin{lemma}\label{lem:graph-stable-ideals}
		The right $K$-comodule algebra $(H,\delta)$ has no nonzero proper
		two-sided ideal that is a right $K$-subcomodule. The same holds
		for its deformation $(H,\delta)_\sigma$ as a right
		$K^\sigma$-comodule algebra.
	\end{lemma}
	
	\begin{proof}
		Let $J$ be a nonzero two-sided ideal and a right
		$K$-subcomodule of $(H,\delta)$.
		The Hopf algebra map
		$\pi:K\to H^{\cop}$, given by
		$\pi(u\otimes a)=\eps(u)a$, satisfies
		\[
		(\id\otimes\pi)\delta(h)=h_2\otimes h_1.
		\]
		Hence $\Delta_H(J)\subseteq H\otimes J$.
		The counit identity implies $\eps|_J\neq0$.
		Choose $j\in J$ with $\eps(j)\neq0$.
		Since $J$ is a left ideal, $S(j_1)j_2=\eps(j)1\in J$, 
		so $J=H$. The assertion for the deformation
		follows from Lemma~\ref{lem:deformation-ideals}.
	\end{proof}

	Set $B=\Gamma_\tau(H)$. 
	Then
	\[
	\Gamma_\tau:(H,\delta)\xrightarrow{\sim}B
	\]
	is an isomorphism in $\Alg(\Comod(K))$. Since $K$ and
	$K^\sigma$ have the same coalgebra, we use the same symbol
	$\delta$ for the corresponding right $K^\sigma$-coaction.
	
	\begin{proposition}[Realization of $H^\imath_\tau$]
		\label{prop:ihopf-deformation}
		$\Gamma_\tau$ induces an isomorphism in
		$\Alg(\Comod(K^\sigma))$
		\begin{equation}\label{eq:main-isomorphism}
			\Gamma_\tau:
			(H^\imath_\tau,\delta)
			\xrightarrow{\;\sim\;}
			B_\sigma .
		\end{equation}
	\end{proposition}
	
	\begin{proof}
		Under the comodule-algebra deformation of the right
		$K$-comodule algebra $(H,\delta)$, the multiplication is
		\begin{align*}
			a\cdot_\sigma b
			&=
			a_2b_2\,
			\sigma^{-1}
			\bigl(
			\tau(a_3)\otimes a_1,
			\tau(b_3)\otimes b_1
			\bigr)\\
			&=
			a_2b_2\,
			\eps(\tau(a_3))
			\phi(\tau(b_3),a_1)
			\eps(b_1)\\
			&=
			\phi(\tau(b_2),a_1)a_2b_1
			=
			a*_\tau b.
		\end{align*}
		Thus $(H,\delta)_\sigma=(H^\imath_\tau,\delta)$.
		Lemma~\ref{lem:coideal-deformation} now gives
		\eqref{eq:main-isomorphism}.
	\end{proof}

	By Lemma~\ref{lem:graph-stable-ideals} and
	Proposition~\ref{prop:ihopf-deformation}, $(H^\imath_\tau,\delta)$
	has no nonzero proper two-sided ideal that is a subcomodule.
	By Proposition~\ref{prop:cocycle-double}, the Hopf $2$-cocycle deformation $K^\sigma$ identifies with the Drinfeld double $D(H)$.
	
	\begin{lemma}[Comparison of coactions]
		\label{lem:coaction-comparison}
		Under the identification $K^\sigma=D(H)$ and the Hopf algebra
		isomorphism
		\[
		\Phi^{\mathrm D}:
		D(H)\xrightarrow{\;\sim\;}(H\otimes H)^\imath
		\]
		of \cite[Proposition~2.8]{CLPRW}, the coaction $\delta$ becomes
		the coaction formula of \cite[Proposition~2.11]{CLPRW}. Namely,
		\begin{equation}\label{eq:main-coaction}
			(\id\otimes\Phi^{\mathrm D})\delta(h)
			=
			\phi(\tau(h_4),h_2)h_3
			\otimes(\tau(h_5)\otimes h_1)
			=
			\Psi(h).
		\end{equation}
	\end{lemma}
	
	\begin{proof}
		Using \eqref{eq:delta} and the formula for $\Phi^{\mathrm D}$,
		\begin{align*}
			(\id\otimes\Phi^{\mathrm D})\delta(h)
			=
			h_2\otimes
			\Phi^{\mathrm D}(\tau(h_3)\otimes h_1)
			=
			\phi(\tau(h_4),h_2)h_3
			\otimes(\tau(h_5)\otimes h_1),
		\end{align*}
		which is the stated formula.
	\end{proof}

	\begin{theorem}[Hopf Galois realization]
		\label{thm:main}
		$\Gamma_\tau$ gives an injective morphism in
		$\Alg(\Comod(K^\sigma))$
		\begin{equation}\label{eq:galois-embedding}
			\Gamma_\tau:
			H^\imath_\tau
			\longrightarrow
			F_\sigma(K).
		\end{equation}
		$F_\sigma(K)$ is a
		$(K,K^\sigma)$-Hopf bi-Galois object, and hence in particular a
		right $K^\sigma$-Galois object.
	\end{theorem}
	
	\begin{proof}
		Apply Lemma~\ref{lem:coideal-deformation} to
		$\Gamma_\tau$ and $\sigma$, using
		Proposition~\ref{prop:ihopf-deformation} to identify the
		deformed algebra with $H^\imath_\tau$.
	\end{proof}

	\begin{corollary}\label{cor:graph-hopf-target}
		For $a,b\in H$, one has
		\[
		\sigma^{-1}(\Gamma_\tau(a),\Gamma_\tau(b))
		=\phi(\tau(b),a).
		\]
		The map $\Gamma_\tau:H^\imath_\tau\to D(H)$ is an algebra map
		if and only if
		\[
		\phi(\tau(b),a)=\eps(a)\eps(b),\qquad a,b\in H.
		\]
	\end{corollary}
	
	\begin{proof}
		Substitute $\Gamma_\tau$ in \eqref{eq:sigma-inverse} and use
		the counit identities. The criterion is
		Proposition~\ref{prop:Gamma-hopf-target}, expressed using
		$\sigma^{-1}$.
	\end{proof}

	\begin{corollary}
		\label{cor:diagonal-realization}
		Transport the right $K^\sigma$-coaction of $F_\sigma(K)$
		along
		\[
		K^\sigma\cong D(H)
		\xrightarrow{\;\Phi^{\mathrm D}\;}
		(H\otimes H)^\imath.
		\]
		Then
		\[
		\Gamma_\tau:
		(H^\imath_\tau,\Psi)
		\longrightarrow
		F_\sigma(K)
		\]
		is an injective morphism in
		$\Alg(\Comod((H\otimes H)^\imath))$.
	\end{corollary}
	
	\begin{proof}
		This follows immediately from
		Theorem~\ref{thm:main} and
		Lemma~\ref{lem:coaction-comparison}.
	\end{proof}
	
	\begin{remark}
		The choice $K=H\otimes H^{\cop}$
		is adapted to the right-comodule convention used here. If the same
		linear map
		\[
		\Gamma_\tau(h)=\tau(h_2)\otimes h_1
		\]
		is instead regarded as taking values in $H^{\cop}\otimes H$,
		its image is a left coideal subalgebra.
	\end{remark}

	\begin{proposition}\label{prop:Gamma-MU-stability}
		The subalgebra $B_\sigma=\Gamma_\tau(H^\imath_\tau)\subseteq F_\sigma(K)$
		is stable under the right Miyashita--Ulbrich action of $K^\sigma$
		if and only if $\tau(H)\subseteq Z(H)$.
		If $\tau$ is injective, this is equivalent to commutativity of $H$.
	\end{proposition}
	
	\begin{proof}
		By Lemma~\ref{lem:MU-stability}, it is enough to determine when
		$B=\Gamma_\tau(H)$ is stable under the right adjoint action of $K$.
		For $u,a,h\in H$,
		\[
		\Gamma_\tau(h)\leftharpoonup(u\otimes a)
		=
		S(u_1)\tau(h_2)u_2
		\otimes
		S^{-1}(a_2)h_1a_1.
		\]
		
		Suppose that $B$ is stable. Since
		$(\eps_H\otimes\id_H)\Gamma_\tau=\id_H$, taking $a=1$ gives
		\[
		S(u_1)\tau(h)u_2
		=
		\eps_H(u)\tau(h).
		\]
		Hence
		\[
		\tau(h)u
		=
		u_1\bigl(S(u_2)\tau(h)u_3\bigr)
		=
		u\tau(h),
		\]
		so $\tau(H)\subseteq Z(H)$.
		
		Conversely, suppose that $\tau(H)\subseteq Z(H)$. Then
		\[
		\Gamma_\tau(h)\leftharpoonup(u\otimes1)
		=
		\eps_H(u)\Gamma_\tau(h),
		\]
		and
		\[
		\begin{aligned}
			\Gamma_\tau(S^{-1}(a_2)ha_1)
			&=
			\tau(S^{-1}(a_3))\tau(h_2)\tau(a_2)
			\otimes S^{-1}(a_4)h_1a_1\\
			&=
			\tau(h_2)\otimes S^{-1}(a_2)h_1a_1.
		\end{aligned}
		\]
		The second equality uses centrality of $\tau(H)$ and the
		antipode identity in $H^{\cop}$. Thus both $H\otimes1$ and
		$1\otimes H^{\cop}$ preserve $B$. These subalgebras generate
		$K$, so $B$ is stable under the right adjoint action of $K$.
		
		If $\tau$ is injective and $\tau(H)\subseteq Z(H)$, then
		$\tau(ab-ba)=0$ for all $a,b\in H$, so $H$ is commutative.
		Conversely, commutativity of $H$ implies $\tau(H)\subseteq Z(H)$.
	\end{proof}

	\section{Compatible maps and algebra characters}
	\label{sec:characters}
	
	We relate the $\tau$-twisted compatible maps of
	Chen--Lu--Pan--Ruan--Wang to algebra characters
	of $H^\imath_\tau$ and their character-evaluation maps.
	
	By Proposition~\ref{prop:ihopf-deformation} and
	Lemma~\ref{lem:coaction-comparison},
	$H^\imath_\tau$ is an algebra object in
	$\Comod((H\otimes H)^\imath)$.
	Thus, for $t\in\Alg(H^\imath_\tau,k)$,
	Proposition~\ref{pro:character} gives a morphism
	\begin{equation}\label{eq:jt-diagonal}
		j_t=(t\otimes\id)\Psi:
		H^\imath_\tau\longrightarrow(H\otimes H)^\imath
	\end{equation}
	in $\Alg(\Comod((H\otimes H)^\imath))$.
	By \eqref{eq:delta},
	\begin{equation}\label{eq:jt}
		(\Phi^{\mathrm D})^{-1}\circ j_t(h)
		=
		t(h_2)\tau(h_3)\otimes h_1.
	\end{equation}
	
	Recall that a linear map $\chi:H\to k$ is
	\emph{$\tau$-twisted compatible} if
	\begin{equation}\label{eq:chi}
		\chi(ab)
		=
		\chi(a_1)\chi(b_2)
		\phi(\tau(a_2),b_1)
	\end{equation}
	for all $a,b\in H$
	\cite[Definition~2.9]{CLPRW}.
	For $\tau=\id$, it is called a compatible map.
	
	By \cite[Propositions~2.10 and~2.11]{CLPRW},
	a $\tau$-twisted compatible map $\chi$ determines
	a multiplicative linear map
	\begin{equation}\label{eq:xi}
		\xi_{\tau,\chi}:
		H^\imath_\tau\longrightarrow(H\otimes H)^\imath,
		\qquad
		\xi_{\tau,\chi}(h)
		=
		\chi(h_2)\tau(h_3)\otimes h_1,
	\end{equation}
	satisfying
	\begin{equation}\label{eq:xi-colinear}
		\Delta^\imath\circ\xi_{\tau,\chi}
		=
		(\xi_{\tau,\chi}\otimes\id)\circ\Psi.
	\end{equation}
	
	If $\chi\neq0$, then $\chi(1)=1$ by
	\eqref{eq:chi} with $b=1$.
	Hence $\xi_{\tau,\chi}$ is unital and
	\[
	\xi_{\tau,\chi}
	\in
	\Hom_{\Alg(\Comod((H\otimes H)^\imath))}
	\bigl(H^\imath_\tau,(H\otimes H)^\imath\bigr).
	\]
	By Proposition~\ref{pro:character}, its
	corresponding character is
	\begin{equation}\label{eq:Tchi}
		t_\chi=\eps^\imath\circ\xi_{\tau,\chi},
	\end{equation}
	and
	\begin{equation}\label{eq:xi-character-evaluation}
		\xi_{\tau,\chi}=(t_\chi\otimes\id)\Psi.
	\end{equation}
	Since
	$\eps^\imath(u\otimes a)=\phi(u,S^{-1}(a))$,
	we obtain
	\begin{equation}\label{eq:tchi}
		t_\chi(h)
		=
		\chi(h_2)
		\phi(\tau(h_3),S^{-1}(h_1)).
	\end{equation}
	
	We now prove that every character of $H^\imath_\tau$ arises in
	this way. For the next two lemmas, put
	$\psi(x,y)=\phi(\tau(x),y)$. By
	\cite[Lemma~2.5]{CLPRW}, $\psi$ is a Hopf pairing.
	
	\begin{lemma}\label{lem:TQ}
		Define $T,Q:H^*\to H^*$ by
		\[
		T(f)(h)=f(h_2)\psi(h_3,S^{-1}(h_1)),
		\qquad
		Q(f)(h)=f(h_2)\psi(h_3,h_1).
		\]
		Then $T$ and $Q$ are inverse linear automorphisms of $H^*$.
	\end{lemma}
	
	\begin{proof}
		Using the Hopf-pairing identities,
		\begin{align*}
			T(Q(f))(h)
			&=
			f(h_3)\psi(h_4,h_2)\psi(h_5,S^{-1}(h_1))\\
			&=
			f(h_3)\psi(h_4,h_2 S^{-1}(h_1))
			=
			f(h),
		\end{align*}
		where
		$h_2 S^{-1}(h_1)\otimes h_3=1\otimes h$.
		Similarly,
		\begin{align*}
			Q(T(f))(h)
			&=
			f(h_3)\psi(h_4,S^{-1}(h_2))\psi(h_5,h_1)\\
			&=
			f(h_3)\psi(h_4,S^{-1}(h_2) h_1)
			=
			f(h),
		\end{align*}
		using
		$S^{-1}(h_2) h_1\otimes h_3=1\otimes h$.
	\end{proof}
	
	\begin{lemma}\label{lem:character-compatible}
		Let $t\in\Alg(H^\imath_\tau,k)$. Then the linear map $\chi_t:H\to k$
		defined by
		\begin{equation}\label{eq:chi-from-character}
			\chi_t(h)
			=
			Q(t)(h)
			=
			t(h_2)\phi(\tau(h_3),h_1)
		\end{equation}
		is a nonzero $\tau$-twisted compatible map.
	\end{lemma}
	
	\begin{proof}
		Since $t$ is a character of $H^\imath_\tau$,
		\begin{equation}\label{eq:t-character-relation}
			t(a)t(b)
			=
			\psi(b_2,a_1)t(a_2b_1).
		\end{equation}
		The Hopf-pairing identities give
		\begin{align*}
			\chi_t(ab)
			&=
			t(a_2b_2)\psi(a_3b_3,a_1b_1)\\
			&=
			t(a_3b_3)
			\psi(a_4,a_1b_1)
			\psi(b_4,a_2b_2).
		\end{align*}
		On the other hand, using
		\eqref{eq:t-character-relation},
		\begin{align*}
			\chi_t(a_1)\chi_t(b_2)\psi(a_2,b_1)
			&=
			t(a_2)t(b_3)
			\psi(a_3,a_1)
			\psi(b_4,b_2)
			\psi(a_4,b_1)\\
			&=
			t(a_3b_3)
			\psi(b_4,a_2)
			\psi(a_4,a_1)
			\psi(b_5,b_2)
			\psi(a_5,b_1)\\
			&=
			t(a_3b_3)
			\psi(a_4,a_1b_1)
			\psi(b_4,a_2b_2).
		\end{align*}
		Hence $\chi_t$ satisfies \eqref{eq:chi}.
		Moreover, $\chi_t(1)=t(1)=1$, so $\chi_t\neq0$.
	\end{proof}
	
	\begin{theorem}[Compatible maps and algebra characters]
		\label{thm:compatible-characters}
		The assignment $\chi\mapsto t_\chi$ is a bijection between the nonzero
		$\tau$-twisted compatible maps $H\to k$ and
		$\Alg(H^\imath_\tau,k)$. Its inverse is
		\[
		t\longmapsto\chi_t,
		\qquad
		\chi_t(h)=t(h_2)\phi(\tau(h_3),h_1).
		\]
		Under this bijection,
		\begin{equation}\label{eq:xi-jt}
			\xi_{\tau,\chi_t}=(t\otimes\id)\Psi.
		\end{equation}
	\end{theorem}
	
	\begin{proof}
		By \eqref{eq:tchi}, the assignment
		$\chi\mapsto t_\chi$ is the restriction of $T$, while
		Lemma~\ref{lem:character-compatible} shows that $Q$ sends
		algebra characters to nonzero $\tau$-twisted compatible maps.
		Since $T$ and $Q$ are inverse by Lemma~\ref{lem:TQ}, these
		assignments are inverse.
		
		For $t\in\Alg(H^\imath_\tau,k)$, one has
		$t_{\chi_t}=T(Q(t))=t$. Hence
		\eqref{eq:xi-character-evaluation} gives
		\eqref{eq:xi-jt}.
	\end{proof}
	
	Combining Theorem~\ref{thm:compatible-characters} with
	Proposition~\ref{pro:character} gives a categorical form of the result.
	
	\begin{corollary}\label{cor:compatible-coideal}
		There is a bijection between nonzero
		$\tau$-twisted compatible maps $H\to k$ and
		\[
		\Hom_{\Alg(\Comod((H\otimes H)^\imath))}
		\bigl(H^\imath_\tau,(H\otimes H)^\imath\bigr).
		\]
		Under this bijection, the morphism corresponding to $\chi$ is
		$\xi_{\tau,\chi}$.
	\end{corollary}
	
	Theorem~\ref{thm:xi-injective} below identifies $H^\imath_\tau$
	with the right coideal subalgebra
	$\xi_{\tau,\chi}(H^\imath_\tau)$ for every nonzero
	$\tau$-twisted compatible map $\chi$. This gives the right coideal
	realization discussed in \cite[Remark~2.12]{CLPRW}, where injectivity
	is assumed.

	\begin{theorem}\label{thm:xi-injective}
		Let $\chi:H\to k$ be a nonzero
		$\tau$-twisted compatible map. Then
		\[
		\xi_{\tau,\chi}:
		H^\imath_\tau\longrightarrow (H\otimes H)^\imath
		\]
		is an injective morphism in
		$\Alg(\Comod((H\otimes H)^\imath))$. Consequently,
		$\xi_{\tau,\chi}(H^\imath_\tau)$ is a right coideal subalgebra of
		$(H\otimes H)^\imath$.
	\end{theorem}
	
	\begin{proof}
		By Theorem~\ref{thm:compatible-characters},
		$\xi_{\tau,\chi}=(t_\chi\otimes\id)\Psi$.
		Proposition~\ref{prop:stable-ideals},
		Lemma~\ref{lem:graph-stable-ideals}, and the coaction comparison
		in Lemma~\ref{lem:coaction-comparison} show that this map is
		injective. Proposition~\ref{pro:character} gives the
		right-colinear algebra-map and right coideal assertions.
	\end{proof}
	
	\begin{proposition}\label{prop:compatible-convolution-inverse}
		Let $\chi:H\to k$ be a nonzero $\tau$-twisted compatible map.
		Using the Hopf pairing $\psi(x,y)=\phi(\tau(x),y)$ as above,
		$\chi$ is convolution-invertible, with inverse
		\[
		\bar\chi(h)
		=
		\chi(S(h_2))\psi(h_1,S(h_3))
		=
		\chi(S(h_2))\phi(\tau(h_1),S(h_3)).
		\]
		Moreover, if
		\[
		W_f(h)=h_1f(h_2),
		\qquad f\in H^*,
		\]
		then $W_\chi$ is a linear automorphism of $H$ with inverse
		$W_{\bar\chi}$, and
		\[
		W_{\bar\chi}\circ(\eps\otimes\id)\circ\xi_{\tau,\chi}
		=\id_H.
		\]
		In particular,
		$W_{\bar\chi}\circ(\eps\otimes\id)$ is a linear left inverse of
		$\xi_{\tau,\chi}$.
	\end{proposition}
	
	\begin{proof}
		Since $\tau S=S\tau$, the Hopf-pairing identity gives
		\[
		\psi(S(x),y)=\psi(x,S(y)).
		\]
		Since $\chi(1)=1$, applying the compatibility relation to
		$h_1S(h_2)$ gives
		\begin{align*}
			(\chi*\bar\chi)(h)
			=
			\chi(h_1)\chi(S(h_3))\psi(h_2,S(h_4))
			=
			\chi(h_1S(h_2))
			=
			\eps(h).
		\end{align*}
		Using $\psi(S(x),y)=\psi(x,S(y))$ and applying the compatibility
		relation to $S(h_1)h_2$, we similarly obtain
		\begin{align*}
			(\bar\chi*\chi)(h)
			=
			\chi(S(h_2))\chi(h_4)\psi(h_1,S(h_3))
			=
			\chi(S(h_2))\chi(h_4)\psi(S(h_1),h_3)
			=
			\chi(S(h_1)h_2)
			=
			\eps(h).
		\end{align*}
		Thus $\bar\chi$ is the two-sided convolution inverse of $\chi$.
		
		For $f,g\in H^*$,
		\[
		W_f\circ W_g=W_{f*g},
		\]
		so $W_\chi$ is a linear automorphism with inverse
		$W_{\bar\chi}$. Finally,
		\[
		((\eps\otimes\id)\circ\xi_{\tau,\chi})(h)
		=
		\chi(h_2)\eps(\tau(h_3))h_1
		=
		h_1\chi(h_2)
		=
		W_\chi(h),
		\]
		and hence
		\[
		W_{\bar\chi}\circ(\eps\otimes\id)\circ\xi_{\tau,\chi}
		=\id_H.
		\]
	\end{proof}

	We next determine the comodule-algebra isomorphisms
	between the $H^\imath_\tau$ and compare the images
	of the maps $\xi_{\tau,\chi}$.
	All convolutions below use $\Delta_H$; for
	$r\in\Alg(H,k)$, write $r^{-1}=r\circ S$.
	
	\begin{proposition}[Comodule-algebra isomorphisms]
		\label{prop:ihopf-isomorphisms}
		Let $\tau,\upsilon\in\EndHopf(H)$.
		The isomorphisms
		$H^\imath_\tau\to H^\imath_\upsilon$
		in $\Alg(\Comod(K^\sigma))$
		are precisely the maps
		\[
		W_r(h)=h_1r(h_2),
		\qquad
		r\in\Alg(H,k),\quad
		\upsilon*r=r*\tau,
		\]
		with $r=\eps\circ W_r$ uniquely determined.
		Consequently,
		$H^\imath_\tau\cong H^\imath_\upsilon$
		in $\Alg(\Comod(K^\sigma))$ if and only if
		$\upsilon=r*\tau*r^{-1}$ for some
		$r\in\Alg(H,k)$.
		Thus the isomorphism classes are the orbits
		of $\Alg(H,k)$ on $\EndHopf(H)$ under
		\[
		r\cdot\tau=r*\tau*r^{-1}.
		\]
	\end{proposition}
	
	\begin{proof}
		Let $Q=H^{\cop}$ and
		$\pi(u\otimes a)=\eps(u)a$. By \eqref{eq:delta},
		\[
		(\id\otimes\pi)\delta_\tau
		=(\id\otimes\pi)\delta_\upsilon
		=\Delta_Q.
		\]
		Since $Q$ and $H$ have the same algebra
		and $\Delta_Q=\Delta_H^{\cop}$,
		the map $U_r$ of
		Lemma~\ref{lem:quotient-coaction-morphisms}
		is $W_r$. That lemma and
		Proposition~\ref{prop:ihopf-deformation}
		show that every isomorphism is $W_r$
		for a unique $r\in\Alg(H,k)$.
		
		Moreover,
		\[
		\begin{aligned}
			\delta_\upsilon W_r(h)
			&=h_2\otimes(\upsilon*r)(h_3)\otimes h_1,\\
			(W_r\otimes\id_K)\delta_\tau(h)
			&=h_2\otimes(r*\tau)(h_3)\otimes h_1.
		\end{aligned}
		\]
		Applying $\eps\otimes\id_H\otimes\eps$
		shows that colinearity is equivalent to
		\[
		\upsilon*r=r*\tau
		\iff \upsilon=r*\tau*r^{-1}.
		\]
		
		For $r\in\Alg(H,k)$, put
		$c_r(h)=r(h_1)h_2r^{-1}(h_3)$.
		Then
		\[
		c_r(ab)=c_r(a)c_r(b),\qquad
		\Delta c_r=(c_r\otimes c_r)\Delta,
		\qquad c_r^{-1}=c_{r^{-1}}.
		\]
		Hence $c_r\in\AutHopf(H)$ and
		\[
		r*\tau*r^{-1}
		=\tau\circ c_r\in\EndHopf(H).
		\]
		Finally,
		\[
		r\cdot(s\cdot\tau)=(r*s)\cdot\tau,
		\qquad \eps\cdot\tau=\tau,
		\]
		proving the orbit assertion.
	\end{proof}
	
	\begin{corollary}[Equality of images]
		\label{cor:compatible-equal-images}
		Let $\tau,\upsilon\in\EndHopf(H)$, and let
		$\chi,\psi:H\to k$ be nonzero $\tau$-twisted
		and $\upsilon$-twisted compatible maps,
		respectively. Then
		\[
		\operatorname{im}\xi_{\tau,\chi}
		=\operatorname{im}\xi_{\upsilon,\psi}
		\]
		if and only if there is a unique
		$r\in\Alg(H,k)$ satisfying
		\begin{equation}\label{eq:compatible-equal-images}
			\upsilon=r*\tau*r^{-1},
			\qquad
			\psi=\chi*r^{-1}.
		\end{equation}
		Moreover, $r=\psi^{-1}*\chi$, where
		$\psi^{-1}$ is the convolution inverse
		from Proposition~
		\ref{prop:compatible-convolution-inverse}.
	\end{corollary}

	\begin{proof}
		By Proposition~\ref{prop:ihopf-deformation},
		Lemma~\ref{lem:coaction-comparison}, and
		Theorem~\ref{thm:compatible-characters},
		the maps $\xi_{\tau,\chi}$ and
		$\xi_{\upsilon,\psi}$ correspond to
		character-evaluation maps under
		$\Phi^{\mathrm D}$.
		
		By Theorem~\ref{thm:character-evaluation-images}
		and Proposition~\ref{prop:ihopf-isomorphisms},
		their images coincide if and only if there is
		$r\in\Alg(H,k)$ satisfying
		\[
		\upsilon*r=r*\tau,\qquad
		\xi_{\upsilon,\psi}W_r=\xi_{\tau,\chi}.
		\]
		Since
		\[
		(\eps_H\otimes\id_H)\xi_{\tau,\chi}=W_\chi,
		\qquad
		(\eps_H\otimes\id_H)\xi_{\upsilon,\psi}=W_\psi,
		\]
		the second equality implies
		$W_\psi W_r=W_\chi$, hence $\psi*r=\chi$.
		
		Conversely, if $\upsilon*r=r*\tau$ and
		$\psi*r=\chi$, then \eqref{eq:xi} gives
		\[
		\xi_{\upsilon,\psi}W_r(h)
		=(\psi*\upsilon*r)(h_2)\otimes h_1
		=(\chi*\tau)(h_2)\otimes h_1
		=\xi_{\tau,\chi}(h).
		\]
		Thus the conditions are equivalent to
		\eqref{eq:compatible-equal-images}.
		Finally, Proposition~
		\ref{prop:compatible-convolution-inverse}
		gives $r=\psi^{-1}*\chi$.
	\end{proof}

	We now turn to the Borel specialization of
	\cite[Section~4.1]{CLPRW}. Work over
	$k=\mathbb Q(v^{1/2})$, with $v$ an indeterminate, and
	take the Cartan matrix $C=(c_{ij})_{i,j\in I}$ and the
	positive integral symmetrizer $D=\operatorname{diag}(d_i)$
	of that construction. Put $v_i=v^{d_i}$.
	
	Let $H=\widetilde{\mathbf B}$ be the corresponding Hopf
	algebra, generated by $h_i^{\pm1}$ and $\vartheta_i$
	$(i\in I)$, with the Hopf pairing $\phi$ of
	\cite[(4.7)]{CLPRW}. Put $K=H\otimes H^{\cop}$, and let
	$\sigma$ be the cocycle of \eqref{eq:sigma}.
	
	We use the relations and coproducts
	\[
	h_i\vartheta_j
	=v_i^{c_{ij}}\vartheta_jh_i,
	\qquad
	\Delta(h_i)=h_i\otimes h_i,
	\qquad
	\Delta(\vartheta_i)
	=\vartheta_i\otimes1+h_i\otimes\vartheta_i.
	\]
	Here $v_i^2\neq1$, the $h_i$ are algebraically
	independent, and $\vartheta_i\neq0$.
	
	Let $\tau,\upsilon$ be diagram involutions preserving
	$C$ and $D$, acting by
	\[
	\alpha(h_i)=h_{\alpha(i)},\qquad
	\alpha(\vartheta_i)=\vartheta_{\alpha(i)},
	\qquad \alpha\in\{\tau,\upsilon\}.
	\]

	\begin{remark}
		In this setting, Chen--Lu--Pan--Ruan--Wang
		construct the relevant $\tau$-twisted compatible
		map in \cite[Lemma~4.2]{CLPRW}, establish the
		corresponding algebra embedding in
		\cite[Lemma~4.3]{CLPRW}, and identify the
		$\imath$-Hopf algebra on the Borel with the
		quasi-split universal $\imath$-quantum group
		in \cite[Theorem~4.4]{CLPRW}.
	\end{remark}

	\begin{corollary}\label{cor:borel-specialization}
		In the Borel setting above, we have
		\begin{enumerate}[label=(\roman*)]
			\item In $\Alg(\Comod(K^\sigma))$,
			$H^\imath_\tau\cong H^\imath_\upsilon$
			if and only if $\tau=\upsilon$.
			
			\item
			$\Aut_{\Alg(\Comod(K^\sigma))}
			(H^\imath_\tau)=\{\id_H\}$.
			
			\item For nonzero $\tau$- and
			$\upsilon$-twisted compatible maps
			$\chi,\psi$, respectively,
			\[
			\operatorname{im}\xi_{\tau,\chi}
			=\operatorname{im}\xi_{\upsilon,\psi}
			\iff \tau=\upsilon,\quad\chi=\psi.
			\]
		\end{enumerate}
	\end{corollary}

	\begin{proof}
		Let $r\in\Alg(H,k)$ and put
		$\lambda_i=r(h_i)\in k^\times$.
		Since $h_i\vartheta_i=v_i^2\vartheta_i h_i$ and $v_i^2\neq1$, 
		we have $r(\vartheta_i)=0$.
		The coproduct formulas above, together with
		$r*r^{-1}=\eps_H$, give
		\[
		c_r(h_i)=h_i,\qquad
		c_r(\vartheta_i)=\lambda_i\vartheta_i,
		\qquad c_r=r*\id_H*r^{-1}.
		\]
		Hence
		\[
		c_r=\id_H
		\iff \lambda_i=1\quad\text{for all }i
		\iff r=\eps_H.
		\]
		
		By Proposition~\ref{prop:ihopf-isomorphisms},
		an isomorphism in (i) requires
		$\upsilon=\tau\circ c_r$.
		Comparing the images of $h_i$ gives
		$\upsilon=\tau$, since both are diagram
		involutions. Hence $c_r=\id_H$ and
		$r=\eps_H$. The converse is immediate.
		
		Taking $\upsilon=\tau$ proves (ii).
		Finally, Corollary~\ref{cor:compatible-equal-images}
		gives (iii), since its character $r$ must be
		$\eps_H$.
	\end{proof}

	\section*{Acknowledgments}
	
	I thank Prof. Xin Fang for a helpful conversation on quantum groups.
	
\end{document}